\documentclass[11pt,reqno]{amsart}
\usepackage{amscd,amsmath,amsopn,amssymb,amsthm,multicol}
\usepackage{hyperref}
\usepackage[all]{xy}
\usepackage{amscd}
\usepackage{graphicx}

\DeclareMathOperator{\Ad}{Ad}
\DeclareMathOperator{\ad}{ad}

\DeclareMathOperator{\Exp}{exp}

\DeclareMathOperator{\Id}{Id}

\DeclareMathOperator{\Ric}{Ric}

\DeclareMathOperator{\Span}{span}
\DeclareMathOperator{\rnk}{rk}
\newcommand{\fr}{\mathfrak}
\newcommand{\al}{\alpha}
\newcommand{\be}{\beta}

\newcommand{\bb}{\mathbb}

\newcommand{\thickline}{\noalign{\hrule height 1pt}}
\newtheorem{lemma} {Lemma} [section]
\newtheorem{theorem}[lemma]{Theorem} 
\newtheorem{remark}[lemma] {Remark} 
\newtheorem{prop} [lemma]{Proposition}

\begin{document}
\today 

\title
{Homogeneous Einstein metrics  on  $G_2/T$}

\author{Andreas Arvanitoyeorgos, Ioannis Chrysikos, and Yusuke Sakane}
\address{University of Patras, Department of Mathematics, GR-26500 Rion, Greece}
\email{arvanito@math.upatras.gr}
\email{xrysikos@master.math.upatras.gr}
 \address{Osaka University, Department of Pure and Applied Mathematics, Graduate School of Information Science and Technology,
Osaka 560-043, Japan}
 \email{sakane@math.sci.osaka-u.ac.jp}
\medskip
\noindent
\thanks{The third author was supported by Grant-in-Aid
for Scientific Research (C) 21540080}

\begin{abstract}
   We construct the Einstein equation for an  invariant Riemannian metric on the exceptional full flag manifold $M=G_2/T$. By computing a Gr\"obner basis for a system of polynomials of multi-variables we prove  that this manifold admits exactly two non-K\"ahler invariant Einstein metrics.  Thus $G_2/T$ turns out to be   the first  known example of an exceptional full flag manifold which admits at least one  non-K\"ahler and not normal homogeneous Einstein metric.



\end{abstract}
 
\maketitle

  
  \section*{Introduction}
\markboth{Andreas Arvanitoyeorgos, Ioannis Chrysikos and Yusuke Sakane}{Homogeneous Einstein metrics  on  $G_2/T$ }

A  Riemannian manifold  $(M, g)$  is called Einstein if the metric $g$ 
has constant  Ricci curvature, that is  $\Ric_{g}=\lambda g$ for some
 $\lambda\in\mathbb{R}$, where   $\Ric_{g}$ is the Ricci tensor corresponding to $g$. 
   The question whether $M$ carries an Einstein  metric, and  if so, 
   how many, is a fundamental one in Riemannian geometry.   A number of 
   interesting results in geometry have been motivated and inspired by this hard problem.
     The Einstein equation is a non linear second order PDE, and a good
      understanding of its solutions in the general case seems far from 
      being attained.  It becomes more manageable in the  homogeneous setting. 
       Most known examples of compact simply connected Einstein manifolds
        are homogeneous.  In the homogeneous case the Einstein equation 
         reduces to a system of  algebraic equations for which we are 
         looking for positive solutions. For some cases such solutions 
         can been obtained explicity.  In the compact case, invariant 
         Einstein metrics are also characterized as the critical points 
         of the scalar curvature functional on the space of   invariant
          Riemannian metrics of fixed volume.  We refer to \cite{NRS}  
          and the references therein for more details 
          in compact homogeneous Einstein manifolds.
 
 Let $K$ be a  compact, connected and semisimple Lie group.  	A full flag 
 manifold is a compact homogeneous spaces of the form $K/T$ where $T$  
 is a maximal torus in $K$.  It is well known (\cite[p. ~504]{B}) that such 
 a space admits   $|W(K)|/2$ invariant complex structures (here $W(K)$ is 
 the Weyl group of $K$), which are all equivalent  under an 
 automorphism of $K$ (\cite[p.~ 57]{N}). Also, $K/T$  admits a  unique 
 (up to isometry) $K$-invariant K\"ahler-Einstein metric.      
 
 Non-K\"ahler homogeneous Einstein metrics on full flag manifolds 
 corresponding to classical Lie groups have been  studied by several 
 authors (cf. \cite{Arv}, \cite{Sak}, \cite{Neg}).  Although various
 existence results  of homogeneous Einstein metrics on these spaces have been obtained, 
 the classification  of such metrics   is a 
 demanding task  which remains widely  open. 	In the present paper we study the 
 classification problem of homogeneous Einstein metrics on the full flag manifold $G_2/T$.  
  The  isotropy representation of $G_2/T$  decomposes into six inequivalent 
  irreducible submodules, because the root system of the Lie algebra $\fr{g}_2$
   of $G_2$, has six positive roots (see Section  \ref{section2}).  
    There are three (non isomorphic) 
   flag manifolds corresponding to the exceptional Lie group $G_2$, since there are exactly three different 
   ways to paint black the  simple roots in Dynkin diagram of $\fr{g}_2$, as shown in Figure 1
    (for the classification of generalized flag manifolds in terms of painted Dynkin diagrams see \cite{AA}, \cite{AP}.
 \medskip
 \[
    \begin{picture}(29,0) (102, 2.9)
\put(105,5){\circle*{4.4}}
\put(105, 11.3){\makebox(0,0){$\al_1$}}
\put(106.8, 6.3){\line(1,0){13.2}}
\put(106.8, 5.3){\line(1,0){16.1}}
\put(106.8, 4.3){\line(1,0){12.9}}
\put(117, 3.2){\scriptsize $>$}
\put(125,5){\circle{4}}
\put(125.5, 11.3){\makebox(0,0){$\al_2$}}
\put(114, -9){\makebox(0,0){{ \ $G_{2}(\al_2)$}}}
\end{picture}  \qquad \quad\ \ \ 
 \begin{picture}(29,0) (102, 2.9)
\put(105,5){\circle{4.4}}
\put(105, 11.3){\makebox(0,0){$\al_1$}}
\put(106.8, 6.3){\line(1,0){13.2}}
\put(106.8, 5.3){\line(1,0){16.1}}
\put(106.8, 4.3){\line(1,0){12.9}}
\put(117, 3.2){\scriptsize $>$}
\put(125,5){\circle*{4}}
\put(125.5, 11.3){\makebox(0,0){$\al_2$}}
\put(114, -9){\makebox(0,0){{ \ $G_{2}(\al_1)$}}}
\end{picture}  \qquad \quad \ \ \ 
 \begin{picture}(29,0) (102, 2.9)
\put(105,5){\circle*{4.4}}
\put(105, 11.3){\makebox(0,0){$\al_1$}}
\put(106.8, 6.3){\line(1,0){13.2}}
\put(106.8, 5.3){\line(1,0){16.1}}
\put(106.8, 4.3){\line(1,0){12.9}}
\put(117, 3.2){\scriptsize $>$}
\put(125,5){\circle*{4}}
\put(125.5, 11.3){\makebox(0,0){$\al_2$}}
\put(114, -9){\makebox(0,0){ { \ $G_2/T$}}}
\end{picture}
 \]
 
 \bigskip
 \begin{center}
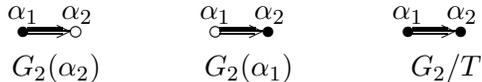
 {\bf Figure 1.} {\small  The painted Dynkin diagrams corresponding to  $G_2$}
 \end{center} 
 
 \noindent 
 If we paint black  one simple root in the Dynkin diagram of $G_2$, 
 then we obtain a flag manifold of the form $G_2/U(2)$  with two or three 
 isotropy summands,  depending on the height of this simple root.  Recall 
 that for the root system of $\fr{g}_2$, we can choose a set of simple roots 
 by $\Pi_{M}=\{\al_1, \al_2\}$ with  $(\al_1, \al_1)=3(\al_2, \al_2)$ and then, 
 the heighest root has the form $\widetilde{\al}=2\al_1+3\al_2$ (see  Section 5). 
  Thus, the flag manifold $G_{2}(\al_2)$  in Figure 1 has two isotropy summands, 
  and  $U(2)$ is represented by the 
  short root of $\fr{g}_2$.\footnote{For the notation $G_{2}(\al_2)$ 
  and $G_{2}(\al_1)$ see the article \cite{AA}.}  
  In \cite{Bom} this space is denoted by $G_2/U(2)_3$ (the subscript is the Dynkin index
  of the $SU(2)$ factor in the denominator subgroup).
  For this space, all 
  $G_2$-invariant  Einstein metrics have been obtained explicity in 
  \cite{Sakane}, \cite{Chry1}.  
  The second flag manifold $G_2(\al_1)$ in Figure 1, has three isotropy 
  summands and the isotropy group $U(2)$ is represented by the long root 
  of $\fr{g}_2$.  In \cite{Bom} this is denoted by $G_2/U(2)_1$.
  For for that space, the $G_2$-invariant  Einstein metrics were 
  initially studied in  \cite{Kim}, and later in \cite{Arv}.   
 
 The full flag manifold $G_{2}/T$, where $T=U(1)\times U(1)$ is a maximal
  torus in $G_2$, is obtained by painting black both simple roots in 
  the Dynkin diagram of $G_2$.   According to \cite{Wa1} a full flag 
  manifold $K/T$   is a normal homogeneous Einstein manifold if and 
  only if all roots of $K$  have the same length, and in this case the
   normal   metric of $K/T$ is never K\"ahler. Therefore, if $K$ is 
   an exceptional Lie group  then $K/T$ is a normal homogeneous Einstein 
   manifold if and only if $K\in\{E_6, E_7, E_8\}$, hence $G_2/T$ is 
   not normal.   Our main result is the following:
      
      \medskip
    { \sc{Theorem A.}}
     {\it   The full flag manifold $G_2/T$ admits exactly three $G_2$-invariant
      Einstein metrics (up to isometry). There is a unique K\"ahler-Einstein 
      metric given  (up to scalar) by  $g=(3, 1, 4, 5, 6, 9)$  and the other 
      two are not K\"ahler.  The approximate values of these invariant 
      metrics are given in  Theorem \em{\ref{Einstein}}.}  
     
     \medskip
     As a consequence of Theorem A,  $G_2/T$ 
   is the first  known example of an exceptional full flag
    manifold which admits at least one  non-K\"ahler and not normal 
    homogeneous Einstein metric.   Note that  the isotropy representation 
    of the  full flag manifolds corresponding to the other exceptional 
    Lie groups $F_4, E_6, E_7$, and $E_8$ decomposes into  $24, 36, 63$, 
    and $120$  isotropy summands respectively (see Section  \ref{section1}, Table 1), 
    so searching for new non K\"ahler and not normal  homogeneous 
    Einstein metrics by using traditional techniques seems to be a difficult task.
  Finaly, note that the present work on $G_2/T$ is the 
  first attempt towards the classification of homogeneous
   Einstein metrics on generalized flag manifolds with six 
   isotropy summands.  

The paper is organised as follows:  
   In  Section  \ref{section1}   we recall the  Lie theoretic description 
   of a full flag manifold  $K/T$ of a  
   compact and connected semisimple Lie group $K$, and we study its  
   isotropy representation.  Next,  following the article 
   \cite{Sak} we describe the structure constants of $K/T$ 
   relative to the associated  isotropy decomposition, and 
   we give the expression of the Ricci tensor of a $K$-invariant 
   metric on $K/T$.  In Section \ref{section2} we consider the exceptional 
   full flag manifold $G_2/T$ and we give its Lie theoretic 
   description. Then we construct the Einstein 
   equation for a $G_2$-invariant Riemannian metric.  In the last section,   we 
   give the corresponding polynomial 
   system, and by computing Gr\"obner basis 
   for this system,  we  prove Theorem A and   obtain the full 
   classification of homogeneous Einstein metrics on $G_2/T$.

 	\section{Full flag manifolds}\label{section1}
 	From now on we consider a full flag manifold $K/T$ where
 	 $T$ is a maximal torus of a compact semisimple Lie group $K$. 
 	   We will give a characterization of $K/T$ in terms of root 
 	  system theory, and we will describe some topics of the associated K\"ahler   geometry.
 	  Then, we study the isotropy representation of $K/T$ and we give the expression of the Ricci tensor 
 	  for a $K$-invariant metric on $K/T$.

\subsection{A Lie theoretic description of $K/T$}\label{section1.1}
 Assume that $\dim_{\bb{R}}T=\rnk G=\ell$.  We denote by $\fr{k}$, $\fr{t}$  the Lie algebras of $K$ and 
$T$ respectively, and by  $\fr{k}_{\bb{C}}=\fr{k}\oplus i\fr{k}$,  
 $\fr{t}_{\bb{C}}=\fr{t}\oplus i\fr{t}$, the corresponding  complexifications.  Let $\fr{t}^{*}$ and $\fr{t}_{\bb{C}}^{*}$ be the dual spaces of $\fr{t}$  and $\fr{t}_{\bb{C}}$, respectively. 
The subalgebra  $\fr{t}_{\bb{C}}$ is a Cartan subalgebra of the complex 
semi-simple Lie algebra $\fr{k}_{\bb{C}}$,  and thus we obtain the root space decomposition 
\[ 
\fr{k}_{\bb{C}}=\fr{t}_{\bb{C}}\oplus\sum_{\al\in R}\fr{k}_{\bb{C}}^{\al},
\]
where $R$ is the root system of $\fr{k}_{\bb{C}}$ relative to $\fr{t}_{\bb{C}}$  and 
\[
\fr{k}_{\bb{C}}^{\al}=\{X\in\fr{k}_{\bb{C}} : \ad(H)X=\al(H)X,   \ \mbox{for all} \   H\in\fr{t}_{\bb{C}}\}
\]
 is the root space associated to the root $\al$.    Recall that by $\bb{C}$-linearity, a root $\al\in R$ is completely determined by its restriction to either $\fr{t}$ or $i\fr{t}$.
 Since the Killing form $B$ of $\fr{k}_{\bb{C}}$ is non-degenerate, for any $\lambda\in \fr{t}_{\bb{C}}^{*}$  we define $H_{\lambda}\in i\fr{t}$  
 by the equation  $B(H_{\lambda}, H)=\lambda(H)$ for all $H\in\fr{t}_{\bb{C}}$.  Let   $i\fr{t}^{*}$ denotes the  real linear subspace of $\fr{t}_{\bb{C}}^{*}$  which  consists of  all $\lambda\in\fr{t}_{\bb{C}}^{*}$ such that the restriction $\lambda|_{\fr{t}}$ has values in $i\bb{R}$; the later condition is equivalent to saying that $\lambda|_{i\fr{t}}$ is real valued.  Note that the restriction map $\lambda\mapsto \lambda|_{i\fr{t}}$ defines an isomorphism from $i\fr{t}^{*}$ onto the real linear dual space $(i\fr{t}^{*})$, which allows us to identify these  spaces.  Then, it is well known that  $R$ spans $i\fr{t}^{*}$ and that $R$ is a finite subset of $i\fr{t}^{*}\backslash\{0\}$.  Thus, if  $\al\in R$ then  $\al\in i\fr{t}^{*}$.    
 
  Let $( \ , \ )$ denote the bilinear form on $\fr{t}_{\bb{C}}^{*}$ induced form the Killing form $B$, that is  $(\lambda, \mu)=B(H_{\lambda}, H_{\mu})$, for any $\lambda, \mu\in\fr{t}_{\bb{C}}^{*}$.  Then, since $B$ is negative definite on $\fr{t}$ and positive definite on $i\fr{t}$, the restriction of $( \ , \ )$ on $i\fr{t}^{*}$ is a positive definite inner product. The weight lattice of $\fr{k}_{\bb{C}}$ with respect to $\fr{t}_{\bb{C}}$ is given by 
 \[
 \Lambda=\{\lambda\in i\fr{t}^{*} : \displaystyle\frac{2(\lambda, \al)}{(\al, \al)}\in\bb{Z} \ \  \text{for all} \ \ \al\in R\}.
 \]
    Let $\Pi$ = $\{\alpha^{}_1, \dots, \alpha^{}_\ell\}$
be a simple root system of $R$, and let $R^{+}$ be the set of all positive roots with respect to $\Pi$.   Consider the fundamental weights corresponding to $\Pi$, that is $\Lambda^{}_1, \ldots, \Lambda_{\ell}\in\Lambda$ such that
\begin{equation}\label{fund}
\frac{2(\Lambda^{}_i, \alpha^{}_j)}{(\alpha^{}_j, \alpha^{}_j)} =
\delta^{}_{ij} \qquad (1 \le i, j \le \ell).
\end{equation}
 Then $\{\Lambda_1, \ldots, \Lambda_{\ell}\}$ forms a  $\bb{Z}$-basis for the weight lattice $\Lambda$, and since $i\fr{t}\cong (i\fr{t})^{*}\cong i\fr{t}^{*}$, it is that $i\fr{t}=\sum_{i=1}^{\ell}\bb{R}\Lambda_{i}$. In the weight lattice $\Lambda$ there is a
 distinguished subset $\Lambda^{+}$ given by
 \[
 \Lambda^{+}=\{\lambda\in \Lambda : (\lambda, \al_{i})>0 \ \ \text{for any} \ \ i=1, \dots, \ell\}=\{\lambda\in \Lambda : (\lambda, \al)>0 \ \ \text{for any} \ \ \al\in R^{+}\}.
 \]
 One can see that $\Lambda^{+}$ is the intersection of $\Lambda$ with the  fundamental Weyl chamber corresponding to $\Pi$, given by
 \[
 C(\Pi)=\{\lambda\in i\fr{t}^{*} : (\lambda, \al_{i})>0 \ \ \mbox{for any} \ \ \al_{i}\in\Pi\}.
 \]
   Elements of $\Lambda^{+}$ are usualy called dominant weights relative to $R^{+}$, and any dominant weight can be expressed as a linear combination of the fundamental weights with non negative coefficients.  For example, set 
 \[
 \delta=\displaystyle\frac{1}{2}\sum_{\al\in R^{+}}\al\in i\fr{t}^{*}.
 \]
   Then,  $\delta=\sum_{i=1}^{\ell}\Lambda_{i}$ and thus $\delta\in\Lambda^{+}$ (cf. \cite[p. ~ 144]{Bum}, \cite[p. ~523]{Hel}). 
 
 Set
 \[
 \fr{n}=\sum_{\al\in R^{+}}\fr{k}^{\al}_{\bb{C}}, \qquad  \fr{b}=\fr{t}_{\bb{C}}\oplus\fr{n}.
 \]
 Since   $R^{+}$ is a closed subset of $R$ and  
 $[\fr{t}_{\bb{C}}, \fr{k}^{\al}_{\bb{C}}]\subseteq \fr{k}^{\al}_{\bb{C}}$, 
 it is clear that $\fr{n}$ and $\fr{b}$ are closed under the Lie bracket, 
 and thus they are complex Lie subalgebras of $\fr{k}_{\bb{C}}$. One 
 can easily show that $\fr{n}$ is a  nilpotent ideal of $\fr{k}_{\bb{C}}$. 
   Morever, since $\fr{t}_{\bb{C}}$ is abelian and normalizes $\fr{n}$, 
   we have that $[\fr{b}, \fr{b}]\subset\fr{n}$, and since $\fr{n}$ is 
   nilpotent and hence solvable, it follows that $\fr{b}$ is solvable. 
   In fact, $\fr{b}$ is a maximal solvable  Lie subalgebra of $\fr{k}_{\bb{C}}$, 
   i.e. a Borel subalgebra of $\fr{k}_{\bb{C}}$. 
 
  Let $K_{\bb{C}}$ denote the  complex simply connected semisimple Lie 
  group whose Lie algebra is $\fr{k}_{\bb{C}}$.  Then,  the 
  connected subgroup $B\subset K_{\mathbb C}$ with Lie algebra $\fr{b}$ 
  is  a Borel subgroup of $K_{\bb{C}}$.  Let $N$ be the connected Lie sugroup of $K_{\bb{C}}$ 
  corresponding to  $\fr{n}$, and 
    $A$   be the connected subgroup of $T$ corresponding to 
    the abelian Lie subalgebra $\fr{a}=i\fr{t}$. Then,  
     according to the Iwasawa decomposition 
    we have that any $g\in K_{\bb{C}}$ 
    can be expressed as  $g=ank$, where $a\in A$, $n\in N$ and $k\in K$. 
    In particular, the map $A\times N\times K\to K_{\bb{C}}$ is a 
    diffeomorphism. 
     At the Lie algebra level, this means that 
     $\fr{k}_{\bb{C}}=\fr{a}\oplus\fr{n}\oplus\fr{k}$. 
  One can easily show that $\fr{a}\oplus\fr{n}$ is a solvable
   subalgebra which is contained in the Borel subalgebra $\fr{b}$.
      Thus   we have $A\times N\subset B\subset K_{\bb{C}}$ and   we can rewrite 
     the Iwasawa decomposition as $K_{\bb{C}}=B\times K$.  This implies that 
       $K$ acts transitively on $
     K_{\bb{C}}/B$ with  isotropy subgroup  the connected 
     closed subgroup $T=K\cap B\subset K$. 
     Thus $K_{\mathbb C}/B$ = $K/T$ as $C^\infty_{}$-manifolds.

        Since $K_{\bb{C}}$ is a complex Lie group 
and $B$ a closed complex subgroup, the quotient 
$K_{\mathbb C}/B$ admits a $K$-invariant complex structure. 
 Furthermore (cf. \cite{B}),  
 the $K$-invariant complex structures on $K_{\bb{C}}/B=K/T$ are in  1-1 correspondence with 
 different choices of positive roots for $\fr{k}_{\bb{C}}$.  
 Since the Weyl group $W(R)$ of the root system of $\fr{k}_{\bb{C}}$ acts transitively on the sets of systems of positive roots, 
 all these complex structures are equivalent.  Moreover, the following holds:

 \begin{theorem}\label{ke}\textnormal{(\cite{B}, \cite{Tak})}.
 There is a 1-1 correspondence between $K$-invariant K\"ahler metrics on $K_{\bb{C}}/B$ and dominant weights in $\Lambda^{+}$.  In particular, the $K$-invariant K\"ahler metric on $K_{\bb{C}}/B$ corresponding to $2\delta$ is a K\"ahler-Einstein metric.
  \end{theorem}
 
 \subsection{The isotropy representation of $K/T$}
 We will now examine the isotropy representation of a full flag manifold $K_{\bb{C}}/B=K/T$.   Cosnider the reductive decomposition $\fr{k}=\fr{t}\oplus\fr{m}$ of $\fr{k}$ with respect to the negative of the Killing form $Q=-B( \ , \ )$, that is 
 $\fr{m}=\fr{t}^{\perp}$ and $\Ad(T)\fr{m}\subset\fr{m}$. As usual, we identify  $\fr{m}=T_{o}(K/T)$ (where $o=eT$ is the identity coset of $K/T$), via the isomorphism
 \[
 \fr{m}\ni X \ \leftrightarrow X^{*}_{o}=\frac{d}{dt}\big\{\Ad(\Exp tX)o\big\}\Big|_{t=0}\in T_{o}(K/T) .
 \]
  Take a Weyl basis $\{H_{\al_{1}}, \ldots, H_{\al_{\ell}}\}\cup\{E_{\al}\in\fr{k}_{\bb{C}}^{\al} : \al\in R\}$ with $B(E_{\al}, E_{-\al})=-1$, $[E_{\al}, E_{-\al}]=-H_{\al}$ and  
    \begin{equation}\label{str}
 [E_{\al}, E_{\be}]=
 \left\{
\begin{array}{ll}
  N_{\al, \be}E_{\al+\be}  & \mbox{if} \ \ \al, \be,\al+\be\in R \\
  0 & \mbox{if} \ \ \al, \be\in R,  \al+\be\notin R.
\end{array} \right.
\end{equation}  
 The numbers $N_{\al, \be}\in\bb{R}$ are called the structure constants of $\fr{k}_{\bb{C}}$ with repsect to $\fr{t}_{\bb{C}}$  and they are such that $N_{\al, \be}=0$  if $\al, \be \in R$, $\al+\be\notin R$, and $N_{\al, \be}=-N_{\be, \al}$,  $ N_{\al, \be}=N_{-\al, -\be}\in \mathbb{R}$ if 
$\al, \be, \al+\be\in R$. Then, the real subalgebra  $\fr{k}$ is  given by 
\begin{equation}\label{realk}
\fr{k}=\sum_{j=1}^{\ell}\bb{R}iH_{\al_{j}}\oplus\sum_{\al\in R^{+}}(\mathbb{R}A_{\al}+\mathbb{R}B_{\al}) =\fr{t}\oplus \sum_{\al\in R^{+}}(\mathbb{R}A_{\al}+\mathbb{R}B_{\al}),
\end{equation}
  where  $A_{\al}=E_{\al}+E_{-\al}$ and  $B_{\al}=i(E_{\al}-E_{-\al})$, $(\al\in R^{+})$.   Note that $\fr{k}$, as a real form of $\fr{k}_{\bb{C}}$, is the fixed point set of the conjugation   $\tau : \fr{k}_{\bb{C}}\to\fr{k}_{\bb{C}}$, defined by $\tau(E_{\al})=E_{-\al}$.
  Since $\fr{t}=\Span_{\bb{R}}\{iH_{\al_{j}} : 1\leq j\leq \ell\}$, then the reductive decomposition $\fr{g}=\fr{t}\oplus\fr{m}$ implies that
  \begin{equation}\label{full}
  \fr{m}=T_{o}(K/T)=\sum_{\al\in R^{+}}(\mathbb{R}A_{\al}+\mathbb{R}B_{\al}).
  \end{equation}
  
  Set $\fr{m}_{\al}=\bb{R}A_{\al}+\bb{R}B_{\al}$ for any $\al\in R^{+}$.   
 The linear space $\fr{m}_{\al}$ is an irreducible $\Ad(T)$-module which does not depend on the choise of an ordering in $R$.  Furthermore, since the roots of $\fr{k}_{\bb{C}}$ with respect to $\fr{t}_{\bb{C}}$ are distinct, and the root spaces are one-dimensional, it  is obvious that $\fr{m}_{\al}\ncong\fr{m}_{\be}$ as $\Ad(T)$-representations, for any two roots $\al, \be\in R^{+}$.  Thus, by using (\ref{full}) we obtain the following:

  \begin{prop}\label{numberfull}
 Let  $M=K/T$ be a full flag manifold  of a compact simple Lie group $K$.  Then the isotropy representation of $M$ decomposes into a direct sum of 2-dimensional pairwise inequivalent irreducible $T$-submodules $\fr{m}_{\al}$, as follows:
 \begin{equation}\label{fullflag}
 \fr{m}=\sum_{\al\in R^{+}}\fr{m}_{\al}.
 \end{equation}
   The number of these submodules is equal to  the cardinality $|R^{+}|$.  
 \end{prop}
  
  In Table 1, following \cite{Brb}, we give for any full flag manifold $K/T$ of a compact simple Lie group $K$  the number  of the corresponding isotropy summands.
  
 \begin{center}
 {\bf Table 1.} \ {\small The number of the isotropy summnads for a full flag manifold $K/T$}
 \end{center}
 \smallskip
{\footnotesize{ \begin{center}
 \begin{tabular}{llll}
 \hline 
 $\mbox{Simple  Lie group} \ G$ & Full flag manifold   $K/T$ & $|R|$ & $\fr{m}=\oplus_{i=1}^{s}\fr{m}_{i}$ \\
  \thickline
  $SU(\ell+1), \ \ell\geq 1$ & $SU(\ell+1)/T$ & $\ell(\ell+1)$ & $s=\ell(\ell+1)/2$ \\
  $SO(2\ell+1), \ \ell\geq 2$ & $SO(2\ell+1)/T$ & $2\ell^{2}$ & $s=\ell^{2}$ \\
  $Sp(\ell), \ \ell\geq 2$ & $Sp(\ell)/T$ & $2\ell^{2}$ & $s=\ell^{2}$ \\
  $SO(2\ell), \ \ell\geq 3$ & $SO(2\ell)/T$ & $2\ell(\ell-1)$ & $s=\ell(\ell-1)$ \\
  $G_2$ & $G_2/T$ & $12$ & $s=6$ \\
  $F_4$ & $F_4/T$ & $48$ & $s=24$ \\
  $E_6$ & $E_6/T$ & $72$ & $s=36$ \\
  $E_7$ & $E_7/T$ & $126$ & $s=63$ \\
  $E_8$ & $E_8/T$ & $240$ & $s=120$ \\
  \hline 
  \end{tabular} 
 \end{center}}}
  
  Note that for $\ell=1$  the full flag $SU(\ell+1)/T$   is 
  $SU(2)/U(1)\cong \bb{C}P^{1}$, which is an isotropy irreducible Hermitian symmetric space.

 \subsection{The Ricci tensor for a $K$-invariant metric on $K/T$}\label{section1.3}  
 Since $K/T$ is a reductive homogeneous space, there is a natural 1-1 correspondence between $K$-invariant symmetric covariant 2-tensors on $K/T$ and $\Ad(T)$-invariant symmetric bilinear forms on $\fr{m}$.  For example, in this correspondence a $K$-invariant Riemannian metric $g$ on $K/T$ corresponds to an $\Ad(T)$-invariant inner product $\langle \ , \ \rangle$ on $\fr{m}$. In particular, since $\fr{m}$ admits the decomposition (\ref{fullflag}) and the $\Ad(T)$-modules are mutually inequivalent, the space of $K$-invariant Riemannian metrics on $K/T$   is given by
 \begin{equation}\label{Inva}
\Big\{  g=\langle \ , \ \rangle=\sum_{\al\in R^{+}}x_{\al}\cdot Q|_{\fr{m}_{\al}} : x_{\al}\in \bb{R}^{+}\Big\}.
\end{equation} 
  Notice that the $K$-invariant K\"ahler-Einstein metric on $K_{\bb{C}}/B=K/T$ corresponding to $2 \delta=2 \sum_{i=1}^{\ell}\Lambda_{i}$ is given by
   \begin{equation}\label{kke}
  g_{2 \delta}=\sum_{\al\in R^{+}} 2(\Lambda_1+\cdots+\Lambda_{\ell}, \al)\cdot  Q|_{\fr{m}_{\al}}.
  \end{equation}
  
  Similarly, the Ricci tensor $\Ric_{g}$ of a $K$-invariant metric $g$ on $K/T$, as a $K$-invariant covariant 2-tensor, will be described by an $\Ad(T)$-invariant symmetric bilinear form on $\fr{m}$ given by
  \[
\Ric_{g} = \sum_{\al\in R^{+}} r_\al x_{\al}\cdot Q|_{\fr{m}_{\al}},
\]
where $r_{\al}$ $(\al\in R^{+})$  are the components of the Ricci 
tensor on each module $\fr{m}_{\al}$.  Since $\fr{m}_{\al}\ncong\fr{m}_{\be}$ for any $\al, \be,\in R^{+}$,  it is $\Ric_{g}(\fr{m}_{\al}, \fr{m}_{\be})=0$ (cf. \cite{Wa2}).

The components $r_{\al}$ admit a useful description in terms of the structure constants of $K/T$
associated to the isotropy decomposition (\ref{fullflag}).  For convenience, we give the general  definition of these quantities for a compact homogeneous space $K/L$ of a compact semi-simple Lie group $K$, following \cite{Wa2} and \cite{SP}.  Let $\fr{k}=\fr{l}\oplus\fr{m}$ be a reductive decomposition of $\fr{k}=T_{e}K$ with respect to a  bi-invariant metric $Q$  on   $\fr{k}$. Assume that the isotropy representation $\fr{m}$ of $K/L$ decomposes into $s$ pairwise inequivalent irreducible $\Ad(L)$-modules $\fr{m}_{i}$ as follows $\fr{m}=\fr{m}_1\oplus\cdots\oplus\fr{m}_{s}$.  Choose a $Q$-orthonormal basis   $\{e_{p}\}$  adapted to $\fr{m}=\oplus_{i=1}^{s}\fr{m}_{i}$, that is  $e_{p}\in \fr{m}_{i}$ for some $i$, 
   and $p<q$ if $i<j$ (with $e_{p}\in \fr{m}_{i}$ and $e_{q}\in\fr{m}_{j}$).  Here, $Q$ is a bi-invariant metric on the Lie algebra $\fr{k}=T_{e}K$.
      Let $A_{pq}^{r}= Q([e_{p}, e_{q}], e_{r})$, so 
   that $[e_{p}, e_{q}]_{\fr{m}}=\sum_{\gamma}A_{pq}^{r}e_{r}$, 
   and set   
   \begin{equation}\label{ijk}
   \displaystyle\genfrac{[}{]}{0pt}{}{k}{ij}=\sum(A_{pq}^{r})^{2}=\sum \big(Q([e_{p}, e_{q}], e_{r})\big)^{2},
   \end{equation}
    where the sum is taken over all 
   indices $p, q, r$ with $e_{p}\in \fr{m}_{i}, e_{q}\in\fr{m}_{j}$, and $e_{r}\in\fr{m}_{k}$.
   The triples $\displaystyle\genfrac{[}{]}{0pt}{}{k}{ij}$ are called the {\it structure constants} of $K/L$ with respect the decomposition  $\fr{m}=\oplus_{i=1}^{s}\fr{m}_{i}$ of $\fr{m}$.  According to \cite{Wa2}, the structure constants are independent of the $Q$-orthonormal bases $\{e_{p}\}$, $\{e_{q}\}$ and $\{e_{r}\}$ chosen for 
   $\fr{m}_{i}, \fr{m}_{j}$ and $\fr{m}_{k}$, respectively, but obviously they depend on the choise 
   of the decomposition of $\fr{m}$.  Also, $\displaystyle\genfrac{[}{]}{0pt}{}{k}{ij}$ is    nonnegative, i.e. $\displaystyle{k \brack {ij}}\geq 0$ with $\displaystyle{k \brack {ij}}=0$ if and only if $Q([\fr{m}_{i}, \fr{m}_{j}], \fr{m}_{k})=0$, and they are symmetric in all 
   three entries, that is  $\displaystyle\genfrac{[}{]}{0pt}{}{k}{ij}=\genfrac{[}{]}{0pt}{}{k}{ji}=\genfrac{[}{]}{0pt}{}{j}{ki}$.
  
   We now  return to $K/T$ and study its  structure constants with respect the  $Q$-orthogonal decomposition $\fr{m}=\sum_{\al\in R^{+}}\fr{m}_{\al}$, where $Q=-B( \ , \ )$, and $\fr{m}_{\al}=\bb{R}A_{\al}+\bb{R}B_{\al}$.  Note that we can rewrite the  above splitting of $\fr{m}$  as    $\fr{m}=\fr{m}_1\oplus\cdots\oplus\fr{m}_{s}$,  where $s=|R^{+}|$.   Since $B(E_{\al}, E_{\al})=-1$, one can verify that the  vectors $A_{\al}$ and  $B_{\al}$ are such that $B(A_{\al}, A_{\al})=B(B_{\al}, B_{\al})=-2$ and $B(A_{\al}, B_{\al})=0$.  Therefore, the set
   \begin{equation}\label{fullbass}
 \Big\{X_{\al}=\frac{A_{\al}}{\sqrt{2}}=\frac{E_{\al}+E_{-\al}}{\sqrt{2}},\qquad Y_{\al}=\frac{B_{\al}}{\sqrt{2}}=\frac{i(E_{\al}-E_{-\al})}{\sqrt{2}}: \al\in R^{+}\Big\},
 \end{equation}
 is a   $Q$-orthonormal  basis of $\fr{m}_{\al}$. If we denote for simplicity such a basis by $\{e_{\al}\}=\{X_{\al}, Y_{\al}\}$, then  
 the notation $\displaystyle{ k \brack ij}$ can be rewritten as $\displaystyle{ \gamma \brack \al\be}$, where $\{e_{\al}\}$, $\{e_{\be}\}$ and $\{e_{\gamma}\}$ are the $Q$-orthogonal bases of the modules $\fr{m}_{\al}, \fr{m}_{\be}$, and $\fr{m}_{\gamma}$, respectively.

 Recall  that   if $\al, \be\in R$ such that $\al+\be\neq 0$, then $[\fr{k}^{\al}_{\bb{C}}, \fr{k}^{\be}_{\bb{C}}]=\fr{k}^{\al+\be}_{\bb{C}}$ and  $B(\fr{k}^{\al}_{\bb{C}}, \fr{k}^{\be}_{\bb{C}})=0$ (cf. \cite[p.~ 168]{Hel}). Since $\displaystyle {\gamma \brack \al\be}\neq 0$ if and only if $Q([\fr{m}_{\al}, \fr{m}_{\be}], \fr{m}_{\gamma})\neq 0$, we can easily conclude that $\displaystyle {\gamma \brack \al\be}\neq 0$, if and only if the positive
 roots $\al, \be, \gamma$ are such that $\al+\be-\gamma=0$. 
 Thus, we have   $\displaystyle {\al+\be \brack \al\be}\neq 0$, for any $\al, \be\in R$ such that $\al+\be\in R$. 
 
 By using the above notation,   it can be shown (cf. \cite{Sak})
 that the Ricci component  $r_{\al}$ corresponding to the isotropy summand $\fr{m}_{\al}$ is given by 
  \begin{equation}\label{ricc}
   r_{\al}=\frac{1}{2x_{\al}}+\frac{1}{8}\sum_{\be, \gamma\in R^{+}}\frac{x_{\al}}{x_{\be}x_{\gamma}} {\al \brack \be\gamma} -\frac{1}{4}\sum_{\be, \gamma\in R^{+}}\frac{x_{\gamma}}{x_{\al}x_{\be}}{ \gamma \brack \al\be}
 \end{equation}
for any $\al\in R^{+}$. 
Hence,   a $K$-invariant metric (\ref{Inva}) on $K/T$ is an Einstein metric with Einstein constant $k$ if and only if it is a positive real solution of the system
\begin{equation}\label{EinsteinM}
\Big\{ r_\al=k : \al\in R^{+}\Big\}.
\end{equation}

 \begin{prop}\label{abc}
 For a full flag manifold $K/T$ the triples $\displaystyle {\al+\be \brack \al \ \be}$  are given by 
  \begin{equation}\label{Sakane}
  {\al+\be \brack \al \ \be}=2N^{2}_{\al, \be}.
 \end{equation}
 \end{prop}
 \begin{proof}
 By definition   (\ref{ijk}) it is   $\displaystyle {\al+\be \brack \al \ \be}=(A_{\al \be}^{\al+\be})^{2}$, where 
 \begin{eqnarray*}  
 A_{\al\be}^{\al+\be} &=& B([X_{\al}+Y_{\al}, X_{\be}+Y_{\be}], X_{\al+\be}+Y_{\al+\be}) \nonumber \nonumber \\ 
  &=& \frac{1}{2\sqrt{2}}B([A_{\al}, A_{\be}]+[A_{\al}, B_{\be}]+[B_{\al}, A_{\be}]+[B_{\al}, B_{\be}],A_{\al+\be}+B_{\al+\be}). \label{eq}
 \end{eqnarray*}
 By using (\ref{str})  and the relations $N_{\al,\be}=N_{-\al, -\be}$, and $N_{\al, -\be}=N_{-\al, \be}$, we obtain that
\[
 \begin{tabular}{ll}
 $[A_{\al}, A_{\be}] \ =$ &  $N_{\al, \be}(E_{\al+\be}+E_{-(\al+\be)})+N_{\al, -\be}(E_{\al-\be}+E_{-\al+\be}),$  \\   
 $[A_{\al}, B_{\be}] \ =$ &  $iN_{\al, \be}(E_{\al+\be}-E_{-(\al+\be)})-iN_{\al, -\be}(E_{\al-\be}-E_{-\al+\be}),$  \\ 
 $[B_{\al}, A_{\be}] \ =$ &  $iN_{\al, \be}(E_{\al+\be}-E_{-(\al+\be)})+iN_{\al, -\be}(E_{\al-\be}-E_{-\al+\be}),$  \\ 
 $[B_{\al}, B_{\be}] \ =$ &  $-N_{\al, \be}(E_{\al+\be}+E_{-(\al+\be)})+N_{\al, -\be}(E_{\al-\be}+E_{-\al+\be}),$
 \end{tabular}
 \]
  or equivalently,
 \begin{equation}\label{IMP}
 [A_{\al}+B_{\al},A_{\be}+B_{\be}] = 2iN_{\al, \be}(E_{\al+\be}-E_{-(\al+\be)})+2N_{\al, -\be}(E_{\al-\be}+E_{-\al+\be}). 
 \end{equation}
 But $B(E_{\al}, E_{\be})=0$ if $\al+\be \neq 0$,  and thus $B(E_{\pm(\al-\be)}, E_{\pm(\al+\be)})=0$.  So (\ref{eq}) implies that
  \begin{eqnarray*}
  A_{\al\be}^{\al+\be} &=&  \frac{1}{2\sqrt{2}}B\Big(2iN_{\al, \be}(E_{\al+\be}-E_{-(\al+\be)})+2N_{\al, -\be}(E_{\al-\be}+E_{-\al+\be}), (E_{\al+\be}+E_{-(\al+\be)}) \\ && +i(E_{\al+\be}-E_{-(\al+\be)})\Big) \\
  &=&  \frac{1}{2\sqrt{2}} B\Big(2iN_{\al, \be}(E_{\al+\be}-E_{-(\al+\be)}), (E_{\al+\be}+E_{-(\al+\be)})+i(E_{\al+\be}-E_{-(\al+\be)})\Big) \\
  &=&  \frac{1}{\sqrt{2}}N_{\al, \be}\Bigg(B\Big(i(E_{\al+\be}-E_{-(\al+\be)}), E_{\al+\be}+E_{-(\al+\be)}\Big) \\
  && +B\Big(i(E_{\al+\be}-E_{-(\al+\be)}), i(E_{\al+\be}-E_{-(\al+\be)})\Big)\Bigg) \\
  &=& \frac{1}{\sqrt{2}}N_{\al, \be}\Big(B(B_{\al+\be}, A_{\al+\be})+B(B_{\al+\be}, B_{\al+\be})\Big)  
  = -\frac{2}{\sqrt{2}}N_{\al, \be},  
    \end{eqnarray*}
since $B(B_{\al+\be}, A_{\al+\be})=0$ and $B(B_{\al+\be}, B_{\al+\be})=-2$.  The last equation  implies that $(A_{\al\be}^{\al+\be})^{2}=2N^{2}_{\al, \be}$. 
\end{proof}

\begin{remark}\textnormal{
If $\al, \be\in R$ such that $\al-\be\in R$, then by a similar method we obtain that $\displaystyle {\al-\be \brack \al \ \be}=2N^{2}_{\al, -\be}$}.
\end{remark}  

\begin{remark}\label{W}\textnormal{
 Two roots $\al, \be\in R$ have the same length with respect to the Killing form $B$ 
 if and only if  there is an element $w$ of the Weyl group $W(R)$ of the root system $R$ such that $\be=w(\al)$ (see for example \cite[p. 242]{Tau}).  Thus,  due to the invariance of the Killing form  under $W(R)$,  it is obvious that  for any element $w\in W(R)$ it is}
 \begin{equation}\label{Weyl}
 \genfrac{[}{]}{0pt}{}{w(\gamma)}{w(\al) \ w(\be)}=\genfrac{[}{]}{0pt}{}{\gamma}{\al \ \be}.
 \end{equation}
 \end{remark}
\noindent

  \markboth{Andreas Arvanitoyeorgos, Ioannis Chrysikos and Yusuke Sakane}{Homogeneous Einstein metrics  on  $G_2/T$ }
\section{The full flag manifold $G_2/T$}\label{section2}
\markboth{Andreas Arvanitoyeorgos, Ioannis Chrysikos and Yusuke Sakane}{Homogeneous Einstein metrics  on  $G_2/T$ }  
 We now   study the geometry of the full flag manifold $G_2/T$, where $T$ is a maximal torus of $G_2$. We start by describing  
 its isotropy representation 
 
 \subsection{The decomposition of the isotropy representation of $G_2/T$}\label{section2.1}
The root system of the exceptional  complex simple Lie algebra $\fr{g}_2$  
can be chosen as by $R=\{
\pm\al_{1},\pm \al_{2}, \pm(\al_{1}+\al_{2}), \pm(\al_{1}+2\al_{2}), \pm(\al_{1}+3\al_{2}), \pm(2\al_{1}+3\al_{2})
\}$.  We fix a system of simple roots to be  $\Pi=\{\al_{1}, \al_{2}\}$.  With respect to $\Pi$ the positive roots are given by  
   \begin{equation}\label{positiveG2}
   R^{+}=\{\al_{1}, \al_{2}, \al_{1}+\al_{2}, \al_{1}+2\al_{2}, \al_{1}+3\al_{2}, 2\al_{1}+3\al_{2}\}.
   \end{equation}
    The  maximal root  is $\widetilde{\al}=2\al_1+3\al_2$ (see Figure 2).    The angle between $\al_1$ and $\al_2$ is $5\pi/6$ and  we have $\left\|\al_1\right\|=\sqrt{3}\left\|\al_2\right\|$.  Note that the roots of $G_2$ form succesive angles of $\pi/6$.
    Also, the Weyl group of $G_2$ is generated by rotations of $\mathbb{R}^2$ about the origin through an angle $\pi/6$, and
    reflexions about the vertical axis.
    
    \medskip
     \setlength{\unitlength}{1mm}
          {\footnotesize{
 \begin{picture}(500,30)(70, 10)
 \put(150, 0){\circle{1}}
\put(150, 0){\vector(1,0){20}}
\put(172,2){\makebox(0,0){$\al_2$}}
\put(150,0){\vector(0,1){34}}
\put(150.5, 36){\makebox(0,0){$2\al_1+3\al_2$}}
\put(150,0){\vector(-1,0){20}}
\put(150,0){\vector(-1,2){9}}
\put(140,21){\makebox(0,0){$\al_1+\al_2$}}
 \put(150,0){\vector(-2,1){36}}
 \put(115,21){\makebox(0,0){$\al_1$}}
\put(150,0){\vector(0,-1){34}} 
\put(150,0){\vector(1, 2){9}} 
\put(160,21){\makebox(0,0){$\al_1+2\al_2$}}
\put(150,0){\vector(2,1){36}}
\put(185,21){\makebox(0,0){$\al_1+3\al_2$}}
\put(150,0){\vector(2,-1){36}}
\put(150,0){\vector(1,-2){9}}
\put(150,0){\vector(-2,-1){36}}
\put(150,0){\vector(-1,-2){9}}
  \end{picture}}}

\vspace{4.5cm}
\begin{center}
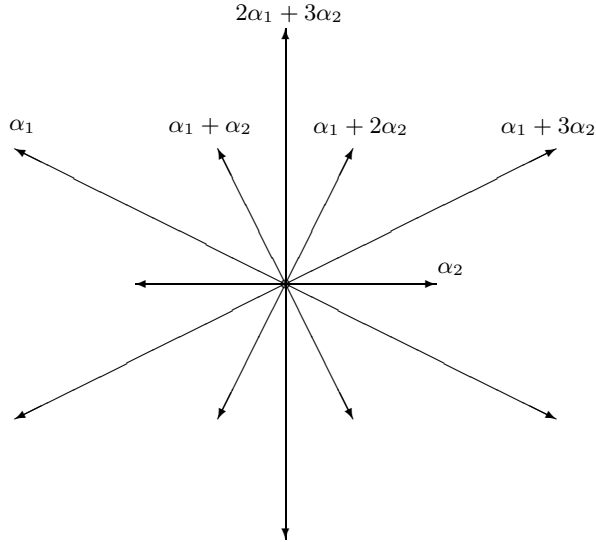
 {{\bf Figure 2}. { \small The root system of $G_2$}}
    \end{center}
 
 The full flag manifold $G_{2}/T$ is obtained by painting black both the two simple roots in the Dynkin diagram of $G_2$.  According to (\ref{fullflag}) and  since    $|R^{+}|=6$, the isotropy representation $\fr{m}$ of $G_2/T$ decomposes into six    inequivalent irreducible $\ad(\fr{k})$- submodules, i.e. 
   \[
   \fr{m}=\fr{m}_1\oplus\fr{m}_2\oplus\fr{m}_3\oplus\fr{m}_4\oplus\fr{m}_5\oplus\fr{m}_6.
   \]
    where  the submodules $\fr{m}_{i}$ $(i=1\leq i \leq 6)$, are given as follows:
 \begin{equation}\label{submodules}
 \left.
 \begin{tabular}{l}
$\fr{m}_1=\fr{m}_{\al_1}=  \mathbb{R}A_{\al_{1}}+\mathbb{R}B_{\al_1}$, \\
$\fr{m}_2=\fr{m}_{\al_2}= \mathbb{R}A_{\al_{2}}+\mathbb{R}B_{\al_2}$, \\ 
$\fr{m}_3=\fr{m}_{\al_1+\al_2}=\mathbb{R}A_{\al_1+\al_2}+\mathbb{R}B_{\al_1+\al_2}$, \\
$\fr{m}_4= \fr{m}_{\al_1+2\al_2} =\mathbb{R}A_{\al_1+2\al_2}+\mathbb{R}B_{\al_1+2\al_2}$,  \\
$\fr{m}_5= \fr{m}_{\al_1+3\al_2}= \mathbb{R}A_{\al_1+3\al_2}+\mathbb{R}B_{\al_1+3\al_2}$,  \\
$\fr{m}_6= \fr{m}_{2\al_1+3\al_2} =\mathbb{R}A_{2\al_1+3\al_2}+\mathbb{R}B_{2\al_1+3\al_2}$.
\end{tabular}\right\}
\end{equation}
\noindent

 \subsection{K\"ahler-Einstein metrics}\label{section2.2} 
 Let $K/T$ be a full flag manifold of a compact connected simple Lie group $K$. A $K$-invariant complex structure on $K/T$ is determined completely by an $\Ad(T)$-invariant endomorphism $J_{o}$ of $\fr{m}_{\mathbb{C}}$ such that $J_{o}^{2}=-\Id_{\fr{m}^{\mathbb{C}}}$ and 
 \begin{equation}\label{Niz}
 [J_{o}X, J_{o}Y]_{\fr{m}^{\mathbb{C}}}-[X, Y]_{\fr{m}^{\mathbb{C}}}-J_{o}[X, J_{o}Y]_{\fr{m}^{\mathbb{C}}}-J_{o}[J_{o}X, Y]_{\fr{m}^{\mathbb{C}}}=0,
 \end{equation}
  for all $X, Y\in\fr{m}_{\mathbb{C}}$. The eigenvalues of $J_{o}$ on $\fr{m}_{\bb{C}}$ are $\pm i$ and thus we obtain the decomposition
  $\fr{m}_{\bb{C}}=\fr{m}_{1,0}\oplus\fr{m}_{0,1}$, where $\fr{m}_{1,0}$ and $\fr{m}_{0,1}$ are the eigenspaces of $J_{o}$   with eigenvalues $+i$ and $-i$, respectively.  One can easily check that the integrability condition (\ref{Niz}) for $J$ is equivalent to the condition that  $\fr{t}_{\bb{C}}\oplus\fr{m}_{1,0}$ is a complex subalgebra of $\fr{k}_{\bb{C}}$. 
  
  Now, from the relation $J_{o}(\Ad(t)X)=\Ad(t)J_{o}(X)$ for any $t\in T$ and $X\in\fr{m}_{\bb{C}}$, one can show that  (cf. \cite{B}, \cite{AP}, \cite{Bor})  
 \begin{equation}\label{complex}
 J_{o}E_{\pm\al}=\pm iE_{\pm\al} \qquad  (\al\in R^{+}).
 \end{equation}
 Thus there is an one-to-one correspondence between invariant complex structures $J$ on $K/T$ and invariant orderings $R^{+}$ in $R$.  Since invariant orderings are in one-to-one correspondence with Weyl chambers, the above bijection is expressed as follows:
 \begin{prop}\label{tcomplex} 
 There is an one-to-one correspondence between invariant complex structures on $K/T$ and  Weyl chambers in $i\fr{t}^{*}$.
 \end{prop} 
 
 In Figure 2, we see that in the root lattice $\Delta_{G_{2}}=\Span_{\bb{Z}}\{\al : \al\in R\}$ of $G_2$, these are twelve Weyl chambers 
 (six positive and six negative), which in turn  induce six different positive orderings $R^{+}$ in $R$,   and six negative ones given by $R^{-}=-R^{+}$. 
 Thus, relation (\ref{complex}) implies that on $G_2/T$,  admits twelve $G_2$-invariant complex structures $J$, or six pairs of conjugate complex structures.  Notice that if $J$ is a complex structure corresponding to an ordering $R^{+}$ in $R$, then the conjugate complex  structure  $\overline{J}$ of $J$, is  determined by the opposite ordering  $R^{-}=-R^{+}$.  Thus we  identify $J$ and $\overline{J}$.   In this way, we conclude that $G_2/T$ admits  six  invariant K\"ahler--Einstein metrics, which are all isometric to each other, since  the six complex structures on $G_2/T$ are    equivalent  under an automorphism of $G_2$.  Any such  metric  is obtained from the other by  permuting the parameters which define it.

We now  compute the (unique) K\"ahler--Einstein metric which is compatible with the natural complex structure
 $J_{\rm nat}$, that is, the complex structure corresponding to the natural invariant ordering $R^{+}$ given by (\ref{positiveG2}).  
 From (\ref{Inva}), a $G_2$-invariant Riemannian metric on $G_2/T$ is given by 
   \begin{equation}\label{metric}
g=  x_1\cdot Q|_{\fr{m}_1}+\cdots+x_6\cdot Q|_{\fr{m}_6}
\end{equation} 
where  we have set $x_1=x_{\al_1}, x_2=x_{\al_{2}}, x_3=x_{\al_1+\al_2}$, $x_4=x_{\al_1+2\al_2}$, $x_5=x_{\al_1+3\al_2}$ and $x_6=x_{2\al_1+3\al_2}$.  Note that  $x_i\in\mathbb{R}^{+}$ for all $i=1, \ldots, 6$.  Next we will denote such metrics by $g=(x_1, x_2, x_3, x_4, x_5, x_6)\in\bb{R_+}^{6}$.

\begin{theorem}\label{G2KE}
The full flag manifold $G_2/T$ admits six  invariant K\"ahler--Einstein metrics which are isometric to each other.  The  K\"ahler-Einstein metric $g_{\rm nat} = g_{2 \delta}^{}$ which is compatible with the natural invariant ordering $J_{\rm nat}$ is given (up to a scale) by $g_{\rm nat}=(3, 1, 4, 5, 6, 9)$.  
\end{theorem}
\begin{proof}
According to notation of Section \ref{section1.1}, the weight $\delta$ for $G_2/T$ is given by 
\[
\delta=\delta_{G_2}=\displaystyle\frac{1}{2}\sum_{\al\in R^{+}}\al=\sum_{i=1}^{2}\Lambda_{i}=\Lambda_{1}+\Lambda_{2},
\]
 where $\Lambda_{1}$ and $\Lambda_2$ are the fundamnetal weights corresponding to the the simples roots $\al_1$ and $\al_2$, respectively.  Now, in Figure 2  one can easily distinguish the long roots 
 \begin{equation}\label{lr}
 \mathcal{L}_{1}=\al_1, \ \ \mathcal{L}_{2}=\al_1+3\al_2, \ \  \mathcal{L}_{3}=2\al_1+3\al_2,
 \end{equation}
  from the short roots 
  \begin{equation}\label{sr}
  \mathcal{S}_{1}=\al_2, \ \  \mathcal{S}_{2}=\al_1+\al_2, \ \  \mathcal{S}_{3}=\al_1+2\al_2.
  \end{equation} 
   It is  $\left\|\mathcal{L}_{i}\right\|=\sqrt{3}\left\|\mathcal{S}_{j}\right\|$ where $1\leq i, j\leq 3$ and $i, j$ independent.   
 We set $(\mathcal{L}_{i}, \mathcal{L}_{i})=3$ and $(\mathcal{S}_{i}, \mathcal{S}_{i})=1$, for any $1\leq i\leq 3$.
 According to expression (\ref{metric}),  the K\"ahler-Einstein metric $g_{\rm nat}$ 
which is compatible   to the natural invariant complex structure $J_{\rm nat}$ defined by the  ordering $R^+$, is given by
{\small {\[
g_{\rm nat}=g_{{\al}_1}\cdot Q|_{\fr{m}_1}+g_{{\al}_2}\cdot Q|_{\fr{m}_2}+g_{({\al}_1+{\al}_2)}\cdot Q|_{\fr{m}_3}
+g_{({\al}_1+2{\al}_2)}\cdot Q|_{\fr{m}_4}+g_{({\al}_1+3e{\al}_2)}\cdot Q|_{\fr{m}_5}+g_{(2{\al}_1+3{\al}_2)}\cdot Q|_{\fr{m}_6},
\]}}
where $\fr{m}_k$ $(k=1, \ldots, 6)$ are given by (\ref{submodules}).  By using (\ref{fund}) and applying  relation (\ref{kke}), we obtain the following values for the components $g_{\al}=(2\delta, \al)$, where $\al\in R^{+}$:
\begin{eqnarray*}
g_{{\al}_1} &=& 2 (\Lambda_1+\Lambda_2, \al_1)= 2 (\Lambda_1, \al_1)=  (\al_1, \al_1)= 3, \\
g_{{\al}_2} &=& 2 (\Lambda_1+\Lambda_2, \al_2)= 2 (\Lambda_2, \al_2)=  (\al_2, \al_2) = 1, \\
g_{{\al}_1+{\al}_2} &=& 2 (\Lambda_1+\Lambda_2, \al_1+\al_2)= 2 (\Lambda_1, \al_1)+2 (\Lambda_2, \al_2)= 4,  \\
g_{{\al}_1+2{\al}_2} &=& 2 (\Lambda_1+\Lambda_2, \al_1+2\al_2)= 2 (\Lambda_1, \al_1)+4 (\Lambda_2, \al_2)= 5, \\ 
g_{{\al}_1+3{\al}_2} &=& 2 (\Lambda_1+\Lambda_2, \al_1+3\al_2)= 2 (\Lambda_1, \al_1)+6 (\Lambda_2, \al_2)= 6, \\
g_{2{\al}_1+3{\al}_2} &=& 2 (\Lambda_1+\Lambda_2, 2\al_1+3\al_2)= 4 (\Lambda_1, \al_1)+6  (\Lambda_2, \al_2)= 9.
\end{eqnarray*}
Thus  the $G_2$-invariant K\"ahler--Einstein metric which corresponds to $J_{\rm nat}$ is given (up to a constant) by $g_{\rm nat}=(3, 1, 4, 5, 6, 9)$.  \end{proof}

 \subsection{Homogeneous Einstein metrics}\label{section2.3}
  We now proceed to the calculation of the Ricci tensor $\Ric_{g}$ corresponding to the $G_2$-invariant metric (\ref{metric}) on $G_2/T$.  
 Following the notation of Section \ref{section1.3}, the tensor $\Ric_{g}$, as a $G_2$-invariant symmetric covariant 2-tensor on $G_2/T$, is given by  
\[
\Ric_{g}=r_{1}x_1\cdot Q|_{\fr{m}_1}+\cdots +r_{6}x_6\cdot Q|_{\fr{m}_6},
\]
where for simplicity we have set   $r_1=r_{\al_1}, r_2=r_{\al_{2}}, r_3=r_{\al_1+\al_2}$, $r_4=r_{\al_1+2\al_2}$, $r_5=r_{\al_1+3\al_2}$ and $r_6=r_{2\al_1+3\al_2}$.  In order to apply    (\ref{ricc}),   we  first need to find  the non zero structure constants  $\displaystyle\genfrac{[}{]}{0pt}{}{k}{ij}$ of $G_2/T$.    For such a procedure we need to determine all  triples of roots $(\al, \be, \gamma)$ with zero sum,  that is $\al+\be+\gamma=0$.  According to Section \ref{section1.3}, by using  relations (\ref{positiveG2}) and (\ref{submodules}),  we obtain the following  results:
\begin{equation}\label{struG2}
{\small{\left.\begin{tabular}{lll}
$\al_1+\al_2+(-(\al_1+\al_2))=0$ & $\Rightarrow$ & $\displaystyle{ \al_1+\al_2 \brack \al_1 \ \al_2}={ 3 \brack 1 2}\neq 0$ \\\\
$\al_2+(\al_1+\al_2)+(-(\al_1+2\al_2))=0$ & $\Rightarrow$ & $\displaystyle{ \al_1+2\al_2 \brack \al_2 \  \al_1+\al_2}={ 4 \brack 2 3}\neq 0$ \\\\
$\al_2+(\al_1+2\al_2)+(-(\al_1+3\al_2))=0$ & $\Rightarrow$ & $\displaystyle{ \al_1+3\al_2 \brack \al_2 \  \al_1+2\al_2}={ 5 \brack 2 4}\neq 0$\\\\
$\al_1+(\al_1+3\al_2)+(-(2\al_1+3\al_2))=0$ & $\Rightarrow$ & $\displaystyle{ 2\al_1+3\al_2 \brack \al_1 \  \al_1+3\al_2}={ 6 \brack 1 5}\neq 0$\\\\
$(\al_1+\al_2)+(\al_1+2\al_2)+(-(2\al_1+3\al_2))=0$ & $\Rightarrow$ & $\displaystyle{ 2\al_1+3\al_2 \brack \al_1+\al_2 \  \al_1+2\al_2}={ 6 \brack 34}\neq 0$
\end{tabular}\right\}}}
\end{equation}

 For the calculation of the above triples we will apply Proposition \ref{abc} and relation (\ref{Weyl}).  Recall  that
\begin{equation}\label{N}
N_{\al, \be}^{2}=N_{\al,\be}N_{-\al, -\be}=\frac{q(p+1)}{2} Q(\al, \al),
\end{equation}
where   $p, q$ are the largest nonnegative integers such that $\be+k\al\in R$, with $-p\leq k\leq q$.

Since, according to  (\ref{lr}) and (\ref{sr}), 
the  positive roots (\ref{positiveG2}) of $G_2$ are divided into long and short ones respectively, we  rewrite (\ref{struG2}) as follows:
 {\small{ \[
  \begin{tabular}{l}
$c_{12}^{3}=\displaystyle\genfrac{[}{]}{0pt}{}{3}{12}= \genfrac{[}{]}{0pt}{}{\al_1+\al_2}{\al_1 \ \al_2}= \genfrac{[}{]}{0pt}{}{\al_1}{\al_2 \ \al_1+\al_2}=\genfrac{[}{]}{0pt}{}{\mathcal{L}_{1}}{\mathcal{S}_{1}  \ \mathcal{S}_{2}}$, \\\\ 
$c_{24}^{5}=\displaystyle\genfrac{[}{]}{0pt}{}{5}{24}= \genfrac{[}{]}{0pt}{}{\al_1+3\al_2}{\al_2 \ \al_1+2\al_2}=\genfrac{[}{]}{0pt}{}{\mathcal{L}_{2}}{\mathcal{S}_{1}  \ \mathcal{S}_{3}}$, \\\\ 
$c_{34}^{6}=\displaystyle\genfrac{[}{]}{0pt}{}{6}{34}= \genfrac{[}{]}{0pt}{}{2\al_1+3\al_2}{\al_1+\al_2 \ \al_1+2\al_2}=\genfrac{[}{]}{0pt}{}{\mathcal{L}_{3}}{\mathcal{S}_{2}  \ \mathcal{S}_{3}}$, \\\\ 
 $c_{15}^{6}=\displaystyle\genfrac{[}{]}{0pt}{}{6}{15}= \genfrac{[}{]}{0pt}{}{2\al_1+3\al_2}{\al_1 \ \al_1+3\al_2}=\genfrac{[}{]}{0pt}{}{\mathcal{L}_{3}}{\mathcal{L}_{1}  \ \mathcal{L}_{2}}$, \\\\ 
$c_{23}^{4}=\displaystyle\genfrac{[}{]}{0pt}{}{4}{23}= \genfrac{[}{]}{0pt}{}{\al_1+2\al_2}{\al_2 \ \al_1+\al_2}=\genfrac{[}{]}{0pt}{}{\mathcal{S}_{3}}{\mathcal{S}_{1}  \ \mathcal{S}_{2}}$.
\end{tabular}
\]}}

\medskip
\noindent Due to Remark \ref{W} and relation (\ref{Weyl}), it is obvious that $\displaystyle\genfrac{[}{]}{0pt}{}{6}{15}\neq  \displaystyle\genfrac{[}{]}{0pt}{}{4}{23}$.  We also obtain the following:
\begin{lemma}\label{three}
The triples $c_{12}^{3}=\displaystyle\genfrac{[}{]}{0pt}{}{3}{12}$, $c_{24}^{5}=\displaystyle\genfrac{[}{]}{0pt}{}{5}{24}$ and $c_{34}^{6}=\displaystyle\genfrac{[}{]}{0pt}{}{6}{34}$ are equal.
\end{lemma}
\begin{proof} 
 The Weyl group $W(R)$ is generated by the simple reflections $\{s_{1}=s_{\al_{1}}, s_{2}=s_{\al_{2}}\}$ defined by
  \begin{equation}\label{refl}
  s_{i}(\al_{j})=\al_{j}-A_{ij}\al_{i}, \qquad (1\leq i, j\leq 2),
  \end{equation}
  where $A_{ij}=\displaystyle\frac{2(\al_{i}, \al_{j})}{(\al_{i}, \al_{i})}$  are the entries of the Cartan matrix of $G_{2}$. The Cartan matrix of $G_2$ (with respect to the fixed basis $\Pi$) is given by 
  $\left( \begin{tabular}{cc}
 2  & -1 \\
-3  &  2 
\end {tabular}
\right)$. We easily get that 
\begin{equation}\label{s}
  s_1(\al_{1})=-\al_{1}, \quad s_{1}(\al_{2})=\al_1+\al_2, \quad s_2(\al_2)=-\al_{2}, \quad s_{2}(\al_{1})=\al_1+3\al_2.
  \end{equation}
  Due to the fact that 
  $\left\|\mathcal{L}_1\right\|=\left\|\mathcal{L}_{2}\right\|=\left\|\mathcal{L}_{3}\right\|$ and 
  $\left\|\mathcal{S}_1\right\|=\left\|\mathcal{S}_{2}\right\|=\left\|\mathcal{S}_{3}\right\|$,
  in order to prove the relation $\displaystyle\genfrac{[}{]}{0pt}{}{3}{12}=\displaystyle\genfrac{[}{]}{0pt}{}{5}{24}$ it suffices to find an element
  $w\in W(R)$ such that $w(\mathcal{L}_1)=\mathcal{L}_{2}$ and $w(\mathcal{S}_{2})=\mathcal{S}_{3}$.  By combining   (\ref{refl}) and (\ref{s}) we obtain that
   $s_{2}(\al_{1})=\al_1+3\al_2$, 
   $s_{2}(\al_1+\al_2)=\al_1+2\al_2$,
  thus $w=s_{2}=s_{\al_{2}}$.  Similarly, for the simple reflection $w'=s_{1}=s_{\al_{1}}$ we compute
  \[
  s_{1}(\mathcal{L}_{2})=s_{1}(\al_1+3\al_2)=2\al_1+3\al_2=\mathcal{L}_{3}, \qquad s_{1}(\mathcal{S}_{1})=s_{1}(\al_{2})=\al_1+\al_2=\mathcal{S}_{2}, 
  \]
  which implies the equality $\displaystyle\genfrac{[}{]}{0pt}{}{5}{24}=\displaystyle\genfrac{[}{]}{0pt}{}{6}{34}$.
  \end{proof}

  We  now proceed   to the calculation of   $\displaystyle\genfrac{[}{]}{0pt}{}{3}{12}$.  From Proposition \ref{abc} we have that 
  \[
  \displaystyle\genfrac{[}{]}{0pt}{}{3}{12}=\displaystyle\genfrac{[}{]}{0pt}{}{\al_1+\al_2}{\al_1 \ \al_2}=2N_{\al_1, \al_2}^{2}.
  \]
   By using the relation $(\al_1, \al_1)=3(\al_2, \al_2)$ and equation (\ref{N}) we get that
   \[
    N_{\al_1, \al_2}^{2}=\displaystyle\frac{3}{2} Q(\al_2, \al_2), \ \ \text{so} \ \ \genfrac{[}{]}{0pt}{}{3}{12}=3 Q(\al_2, \al_2).
   \]
    The normalizing value $(\al_2, \al_2)$ is given by $Q(\al_2, \al_2) = \displaystyle\frac{1}{12}$ (cf. \cite{Brb}) (recall that for the definition of the triples $\displaystyle\genfrac{[}{]}{0pt}{}{\al+\be}{\al \be}$ we used an orthonormal basis of the submodules $\fr{m}_{\al}, \fr{m}_{\be}, \fr{m}_{\al+\be}$). Thus $\displaystyle\genfrac{[}{]}{0pt}{}{3}{12}=\displaystyle\frac{1}{4}$. 
     Similarly, it is
  \[  \genfrac{[}{]}{0pt}{}{4}{23}=\genfrac{[}{]}{0pt}{}{\al_1+2\al_2}{\al_2 \ \al_1+\al_2}=2N_{\al_2, \al_1+\al_2}^{2}=4 Q(\al_2, \al_2)=\frac{1}{3}, 
  \]
   and
  \[
   \genfrac{[}{]}{0pt}{}{6}{15}=  \genfrac{[}{]}{0pt}{}{2\al_1+3\al_2}{\al_1 \ \al_1+3\al_2}=2N_{\al_1, \al_1+3\al_2}^{2}=3 Q(\al_2, \al_2)=\frac{1}{4}.
   \]
   Now, from Lemma \ref{three} we conclude the following:
        \begin{prop}\label{t}
   The non zero triples $\displaystyle\genfrac{[}{]}{0pt}{}{k}{ij}$ of the full flag manifold $G_2/T$ are given by
   \[
  \genfrac{[}{]}{0pt}{}{3}{12}=\genfrac{[}{]}{0pt}{}{5}{24}= \genfrac{[}{]}{0pt}{}{6}{34}=\genfrac{[}{]}{0pt}{}{6}{15}=\frac{1}{4}, \ \ 
  \text{and} \ \ \genfrac{[}{]}{0pt}{}{4}{23}=\frac{1}{3}.
  \]
   \end{prop}
 \noindent  Therefore, we obtain the following expression for the Ricci tensor:
   
   \begin{prop}\label{G}
  The components $r_{i}$ $(i=1, \ldots, 6)$ of the Ricci tensor associated to the  
$G$-invariant Riemannian metric  $g$ given in {\em (\ref{metric})}
 are the following:
 \begin{eqnarray*}
 r_1&=& \frac{1}{2x_1}+  \frac{1}{16}\Big(\frac{x_1}{x_2x_3}- \frac{x_2}{x_1x_3}- \frac{x_3}{x_1 x_2}\Big)+  \frac{1}{16}\Big( \frac{x_1}{x_5x_6}- \frac{x_5}{x_1x_6}- \frac{x_6}{x_1x_5}\Big)  \\
 r_2&=& \frac{1}{2x_2}+  \frac{1}{16}\Big(\frac{x_2}{x_1x_3}- \frac{x_1}{x_2x_3}- \frac{x_3}{x_1 x_2}\Big) +   \frac{1}{12}\Big( \frac{x_2}{x_3x_4}- \frac{x_3}{x_2x_4}- \frac{x_4}{x_2x_3}\Big) + \frac{1}{16}\Big( \frac{x_2}{x_4x_5}- \frac{x_4}{x_2x_5}- \frac{x_5}{x_2x_4}\Big) \\
 r_3&=& \frac{1}{2x_3}+  \frac{1}{16}\Big(\frac{x_3}{x_1x_2}- \frac{x_2}{x_1x_3}- \frac{x_1}{x_2 x_3}\Big) +   \frac{1}{12}\Big( \frac{x_3}{x_2x_4}- \frac{x_2}{x_3x_4}- \frac{x_4}{x_2x_3}\Big) + \frac{1}{16}\Big( \frac{x_3}{x_4x_6}- \frac{x_4}{x_3x_6}- \frac{x_6}{x_3x_4}\Big) \\
 r_4&=& \frac{1}{2x_4}+ \frac{1}{12}\Big(\frac{x_4}{x_2x_3}- \frac{x_2}{x_3x_4}- \frac{x_3}{x_2 x_4}\Big) +   \frac{1}{16}\Big( \frac{x_4}{x_2x_5}- \frac{x_2}{x_4x_5}- \frac{x_5}{x_2x_4}\Big) + \frac{1}{16}\Big( \frac{x_4}{x_3x_6}- \frac{x_3}{x_4x_6}- \frac{x_6}{x_3x_4}\Big) \\
 r_5 &=& \frac{1}{2x_5}+  \frac{1}{16}\Big(\frac{x_5}{x_1x_6}- \frac{x_1}{x_5x_6}- \frac{x_6}{x_1x_5}\Big)+  \frac{1}{16}\Big( \frac{x_5}{x_2x_4}- \frac{x_2}{x_4x_5}- \frac{x_4}{x_2x_5}\Big)  \\
 r_6 &=& \frac{1}{2x_6}+  \frac{1}{16}\Big(\frac{x_6}{x_1x_5}- \frac{x_1}{x_5x_6}- \frac{x_5}{x_1x_6}\Big)+  \frac{1}{16}\Big( \frac{x_6}{x_3x_4}- \frac{x_3}{x_4x_6}- \frac{x_4}{x_3x_6}\Big). 
\end{eqnarray*}
       \end{prop}
  \begin{proof}
  This is a simple consequence of expression (\ref{ricc}) and Proposition \ref{t}.
  \end{proof}
 \medskip  
 According to (\ref{EinsteinM}), a $G_2$-invariant Riemannian metric on the full flag manifold $G_2/T$ is Einstein, if and only if,  there is a positive 
 constant $k$  such that 
 \begin{equation}\label{systemI}
  r_1=k,\quad  r_2=k, \quad r_3=k,\quad r_4=k,\quad r_5=k,\quad r_6=k, 
  \end{equation}
  where $r_{i} \ (i=1, \ldots, 6)$ are given in Proposition \ref{G}.
 
  \section{Proof of Theorem A}\label{section3}
  
  Note that the action of the Weyl group of $G_2$ on the root system of $G_2$ (cf. Figure 2) induces an action on the components of the $G_2$-invariant metric (\ref{metric}). In particular, if 
  $$ ( x_1,  \  x_2 , \ x_3, \ x_4, \ x_5, \  x_6 ) = ( a_1,  \  a_2 , \ a_3, \ a_4, \ a_5, \  x_6 ) $$ is a solution for the system of equations (\ref{systemI}), then 
  $$ ( x_1,  \  x_2 , \ x_3, \ x_4, \ x_5, \  x_6 ) = ( a_5, \  a_2, \ a_4, \  a_3, \ a_1, \  a_6 ) $$ is also a solution of equations (\ref{systemI}). In fact, if $w$ is a reflexion about  $2\alpha _1+ 3\alpha _2$  in the root diagram of $G_2$, then
        $w(\alpha _1) =  \alpha _1+3\alpha _2 $, $w(\alpha _1+\alpha _2) =  \alpha _1+2\alpha _2 $, which induces the action of exchange of $x_1$ to $x_5$ and $x_3$ to $x_4$ respectively and keeping  $x_2$ and $x_6$  fixed. 
        Similarly we see that 
  $$ ( x_1,  x_2 , x_3,  x_4, x_5,  x_6 ) = ( a_6, a_3,  a_4,   a_2, a_1,  a_5 ), \ \  ( x_1,  x_2 , x_3,  x_4, x_5,  x_6 ) = ( a_1,   a_3,  a_2,   a_4,  a_6,  a_5 ) $$ 
    $$ ( x_1,  x_2 , x_3,  x_4, x_5,  x_6 ) = ( a_5,   a_4,  a_2,   a_3, a_6,   a_1 ), \ \ 
  ( x_1,  x_2 , x_3,  x_4, x_5,  x_6 ) =  ( a_6,   a_4, a_3,  a_2, a_5,   a_1 ) $$  
  are also solutions of the equations (\ref{systemI}). These metrics are all isometric
  to each other. 
  
        In order to solve system (\ref{systemI}) we normalize our equations by setting $x_1 = x_5 = 1$, and also $x_4 = x_3$. 
          Then  we obtain the following expression for the Ricci  components in this case: 
 \begin{eqnarray*}
 r_1 = r_5&=& \frac{1}{2} + \frac{1}{16}
   \left(\frac{1}{x_2  x_3}-\frac{x_2}{x_3}-\frac{x_3}{x_2}\right)-\frac{x_6}{16} \\
 r_2&=& \frac{1}{2 x_2}+\frac{1}{12}
   \left(\frac{x_2}{{x_3}^2}-\frac{2}{x_2}\right)+\frac{1}{8}
   \left(\frac{x_2}{x_3}-\frac{x_3}{x_2}-\frac{1}{x_2 x_3}\right) \\
 r_3 = r_4&=&\frac{1}{2 x_3} +\frac{1}{16}
   \left( \frac{x_3}{x_2}-\frac{x_2}{x_3}-\frac{1}{x_2 x_3}\right) -\frac{x_2}{12 {x_3}^2} -\frac{x_6}{16 {x_3}^2} \\
 r_6&=&\frac{1}{2 x_6} +\frac{1}{16}
   \left(x_6 -\frac{2}{x_6}\right)+ \frac{1}{16}
   \left(\frac{x_6}{{x_3}^2}-\frac{2}{x_6}\right). 
\end{eqnarray*}    
   Now the system of equations  (\ref{systemI})  is equivalent to the equations 
    \begin{equation}\label{system2}
  r_1  = r_2, \quad r_2 =  r_3 ,\quad r_3 =  r_6.   
  \end{equation}
  Moreover, we see that  the system of equations  (\ref{system2}) is equivalent to the equations 
  \begin{eqnarray*}
 &  &-9 {x_2}^2 {x_3} - 4 {x_2}^2-3{x_2}
   {x_3}^2 {x_6}+24 {x_2} {x_3}^2+3
   {x_3}^3-16 {x_3}^2+9 {x_3} = 0\\
& &  9 {x_2}^2
   {x_3}+8 {x_2}^2-24 {x_2} {x_3}+3
   {x_2} {x_6}-9 {x_3}^3+16 {x_3}^2-3
   {x_3} = 0\\
 &  &-3 {x_2}^2 {x_3} {x_6}-4
   {x_2}^2 {x_6}-3 {x_2} {x_3}^2
   {x_6}^2-12 {x_2} {x_3}^2+24 {x_2}
   {x_3} {x_6}-6 {x_2} {x_6}^2+3 {x_3}^3
   {x_6}-3 {x_3} {x_6} = 0,  
  \end{eqnarray*} 
  for solutions with $x_2 x_3 x_ 6 \neq 0$. 
  
  For the case  when $x_6 =1$,  we obtain that $x_3 = x_2$ and the equation $15{ x_2}^2 -20 x_2 + 9 = 0$  by computing a Gr\"obner basis, and we do not have positive solution. Thus there are no Einstein metrics for this case. 
  
  For the case  when $x_6 \neq1$,  we obtain the following equations  by computing a Gr\"obner basis : 
    \begin{eqnarray}
   & &  28431 {x_6}^{14}-589032 {x_6}^{13}+5435343
   {x_6}^{12}-29379024 {x_6}^{11}+100757208
   {x_6}^{10} \nonumber\\
   & & -224163176 {x_6}^9 
   +336260186
   {x_6}^8-371473808 {x_6}^7+339968604
   {x_6}^6-262478048 {x_6}^5  \nonumber\\
 & &   +152856152
   {x_6}^4-69550016 {x_6}^3
   +35706576
   {x_6}^2-17407872
   {x_6}+3888000 = 0,  \label{eq29}  
 \end{eqnarray}
   \begin{eqnarray}
  & &    581985310832028473982920358054272527039
   95763069632
   {x_2}  \nonumber
   \\
 & & -36431187984975954069625075825512020735715490145
   97
   {x_6}^{13}  \nonumber
   \\
 & & +729923573884772682153741043747903396277327
   24331072
   {x_6}^{12}  \nonumber
   \\
   & & -646567727758207935002275986628033179230065
   652663061
   {x_6}^{11}  \nonumber
   \\
 & & +332151857904284537155232386064760258410965
   0553728920
   {x_6}^{10}  \nonumber 
    \\
 & & -106305246845145376410007253618095302386497
   40680444344
   {x_6}^9  \nonumber
   \\
 & & +214173648049459114295151905746377531910258398
   27487192
   {x_6}^8  \nonumber
 \\
  & & -283890611717578121261367681274569279647128836
   15920638
   {x_6}^7  \nonumber
      \end{eqnarray}
   \begin{eqnarray}
 & & +283116178659893836079897739452148677823852955
   74349024
   {x_6}^6  \nonumber
\\
 & & -247747049992028930128982437405230731314130824
   14850260
   {x_6}^5  \nonumber
  \\
 & & +175267909611029091296228342932675252979105029
   41466624
   {x_6}^4  \nonumber
 \\ & & -818811448157709557623499817645000761457850005
   6871240
   {x_6}^3  \nonumber
   \\
 & & +356237953427669893952403016556737408987587354
   2732800
   {x_6}^2  \nonumber
   \\
 & & -229895488104401886922601983642414185636336278
   3139696
   {x_6}  \nonumber
   \\
 & & +73815795605614992874388092643016853621308418553
   0880 = 0,    \label{eq30}
    \end{eqnarray}
   \begin{eqnarray}
   & &
   2424938795133451974928834825226135529333156794568
   {x_3}  \nonumber
    \\
 & &+19029972626061774836007867169218828586354518623
   1
   {x_6}^{13}  \nonumber
    \\
 & &-377267218020916490844299704842923123010868
   8015708
   {x_6}^{12}  \nonumber
    \\
 & &+330075960010637578293059366525782191334719
   10058553
   {x_6}^{11}  \nonumber
    \\
 & &-167088331330227007688571325972415637397648
   450592985
   {x_6}^{10}  \nonumber
    \\
 & &+524508423670293884907483441538953074568075
   167613750
   {x_6}^9  \nonumber
    \\
 & &-102864311819049654582348196928443648439258892
   8255299
   {x_6}^8  \nonumber
\\
 & &+132191416807590169058228088475086172695504111
   6133678
   {x_6}^7  \nonumber
    \\
 & &-128682615197266583943369970022392076497298694
   9833794
   {x_6}^6  \nonumber
    \\
 & &+110274796824734249356198080213311353909486877
   8040040
   {x_6}^5  \nonumber
    \\
 & &-748737830066525920856184078153848962211101215
   021298
   {x_6}^4  \nonumber
    \\
 & &+334500258786115622392457312297818354307784956
   975224
   {x_6}^3  \nonumber
    \\
 & &-155759212247584755088196238509822799941866625
   955256
   {x_6}^2  \nonumber
    \\
 & &+954075532838413595542047169961244884467927948
   47168
   {x_6}  \nonumber
    \\
 & &-28083415274725086532725024624855426929207778616
   800 = 0.  \label{eq31}
    \end{eqnarray} 
 
    Now, by  solving   equation (\ref{eq29})   numerically,   we obtain exactly two real solutions  which are approximately given by $ x_6 \approx 
  0.7440$ and  $ x_6 \approx  1.7896$. Substituting these values for $x_6$ into the equations  (\ref{eq30}) and  (\ref{eq31}), we get two real solutions  approximately given by $ x_2 \approx 
  0.2173$,  $ x_3 \approx  1.0234$  and $ x_2 \approx 
  0.2762$,  $ x_3 \approx  1.0347$. Moreover, we obtain the value for $k$ by (\ref{systemI}). 
  Thus we have the following. 
  \begin{theorem}\label{Einstein}
  The full flag manifold $G_2/T$ admits  two non-K\"ahler $G_2$-invariant Einstein metrics.  These metrics are given approximately as follows:  \[
  \begin{tabular}{llllll}
  $x_1 = 1$, & $x_2 \approx 0.2762$, & $x_3 = x_4 \approx1.0347$, & $x_5 = 1$, & $x_6 \approx 1.7896$, & $k \approx 0.3560$ \\
  $x_1 = 1$, & $x_2 \approx 0.2173$, & $x_3 = x_4 \approx1.0234$, & $x_5 = 1$, & $x_6 \approx 0.7440$, & $k \approx 0.4269$.
   \end{tabular}
  \]
  \end{theorem} 
  Note that, for the case  when $ x_5 = x_1 = 1$ and $x_6 \neq1$, we see that $x_3 = x_4$ by computing a Gr\"obner basis.  
  
  \medskip
  
Now we consider the case when   $(x_1 - x_5) (x_1 - x_6) (x_5 - x_6) \neq 0$. 
In this case the system of equations  (\ref{systemI})  is equivalent to the equations 
    \begin{equation}\label{system3}
  r_1  - r_2 = 0, \quad r_2 - r_3 = 0, \quad r_3 -  r_4 = 0, \quad r_4 - r_5 = 0, \quad r_5 -   r_6 = 0.   
  \end{equation}
  Moreover,  by normalizing our equations by setting $x_1 =  1$, we see that  the system of equations  (\ref{system3}) is equivalent to the equations 
   \begin{equation}\label{eqgeneral}
 \left.
 \begin{tabular}{l}
$ \quad  -3 {x_2}^2 {x_3} {x_6}-6{x_2}^2 {x_4} {x_5} {x_6}-4 {x_2}^2 {x_5}{x_6}
  -3 {x_2} {x_3} {x_4} {x_5}^2+24 {x_2} {x_3}{x_4} {x_5} {x_6}
 $  \\
$ -3 {x_2}{x_3} {x_4} {x_6}^2+3{x_2} {x_3} {x_4}+4 {x_3}^2 {x_5} {x_6}+3
   {x_3} {x_4}^2 {x_6}-24
   {x_3} {x_4} {x_5} {x_6} $ \\ 
$ +3 {x_3} {x_5}^2 {x_6}+4 {x_4}^2 {x_5} {x_6}+6 {x_4} {x_5} {x_6} = 0 $,  \\
$ \quad   3{x_2}^2 {x_3} {x_6}+6 {x_2}^2 {x_4} {x_5} {x_6}+8 {x_2}^2 {x_5} {x_6}-3 {x_2} {x_3}^2 {x_5}+3 {x_2} {x_4}^2 {x_5}-24 {x_2} {x_4} {x_5} {x_6} $  \\
$+3 {x_2} {x_5}{x_6}^2 -6 {x_3}^2 {x_4}
   {x_5} {x_6}-8 {x_3}^2
   {x_5} {x_6}-3 {x_3}
   {x_4}^2 {x_6}+24 {x_3}
   {x_4} {x_5} {x_6}-3 {x_3}
   {x_5}^2 {x_6}=0$,  \\
$ \quad 3 {x_2}^2 {x_3} {x_6}-3 {x_2}^2 {x_4} {x_5} {x_6}+6 {x_2}
   {x_3}^2 {x_5}-24 {x_2} {x_3} {x_5} {x_6}-6 {x_2} {x_4}^2 {x_5}+24 {x_2} {x_4} {x_5} {x_6}$  \\
   $ +3{x_3}^2 {x_4} {x_5}{x_6}+8 {x_3}^2 {x_5}
   {x_6}-3 {x_3} {x_4}^2 {x_6}+3 {x_3} {x_5}^2 {x_6}-8 {x_4}^2 {x_5} {x_6}-3 {x_4} {x_5} {x_6} = 0$, \\
   $\quad -4 {x_2}^2 {x_5} {x_6}-3 {x_2} {x_3}^2 {x_5}-3 {x_2} {x_3} {x_4} {x_5}^2+3 {x_2} {x_3}
   {x_4} {x_6}^2-24 {x_2} {x_3} {x_4} {x_6} $ \\
     $ +3 {x_2} {x_3} {x_4}+24 {x_2} {x_3} {x_5} {x_6}+3 {x_2}
   {x_4}^2 {x_5}-3 {x_2} {x_5} {x_6}^2-4 {x_3}^2 {x_5} {x_6}+6 {x_3} {x_4}^2 {x_6}$  \\
      $ -6 {x_3} {x_5}^2 {x_6}+4 {x_4}^2 {x_5} {x_6}= 0$,  \\
     $ \quad  - {x_2}^2 {x_4} {x_5} {x_6}+{x_2}{x_3}^2 {x_5}+8 {x_2}{x_3} {x_4} {x_5} {x_6}
 -8{x_2} {x_3} {x_4} {x_5}-2{x_2} {x_3} {x_4}{x_6}^2+2 {x_2} {x_3}{x_4} $ \\
      $ +{x_2} {x_4}^2 {x_5}-{x_2} {x_5} {x_6}^2-{x_3}^2 {x_4} {x_5} {x_6}+{x_4} {x_5} {x_6} = 0$ 
\end{tabular}\right\}
\end{equation}
    for solutions with $x_2 \, x_3\,   x_4\, x_5 \,x_ 6 \neq 0$. 
    
By computing a Gr\"obner basis for  $(x_1 - x_5) (x_1 - x_6) (x_5 - x_6) \neq 0$, 
  we obtain the following equation for $x_6$ : 
      \begin{eqnarray*}  
 ({x_6}-3) ({x_6}-2) (2 {x_6}-3,) (2 {x_6}-1) (3 {x_6}-2) (3 {x_6}-1)\times \\ 
   (19570190016315141603381713123537334000000000{x_6}^{84}  \\  +472980983168555664602145477565706833680000000{x_6}^{83} \\   -18883997141974881738092861801785333252894800000{x_6}^{82} \\  +74794100641615694951351011555563273206842944000{x_6}^{81} \\ +2178950793290897699447263319032587719040052785696{x_6}^{80} \\   -23920323432867365791709113098968618521666096472928{x_6}^{79} \\ 
  +16203553461279979489585518066332806201459880764320{x_6}^{78} \\ +1146798966623805392526876699318358505159292698566560{x_6}^{77} \\
   -8253365109064021447538266895144369502847176547816896{x_6}^{76} \\
  +15784456131787043608289459347642059765045090310138944{x_6}^{75} \\ 
  +126771120189744874933437937748815475102249501093644584{x_6}^{74} \\
-1206001709632090015722926074131478604699691495585639240{x_6}^{73} 
   \\  +5328035280913750813152995811516544281680919577481605008{x_6}^{72}\\ -12482442118497915371064225556847850469474467891986831088{x_6}^{71}\\ -15498337035862631939271985201881911221051895322143942580{x_6}^{70}\\ +338557565736435688036663787252532332747532110554214415660{x_6}^{69}\\ -2039563706481217288587171432534950085804664817997809025036{x_6}^{68}
\\
     +8633720294376050564114347614772621001678803290761747417964{x_6}^{67}\\ -30615636009083905408974520798029796519322538407366673033114{x_6}^{66}
\\
   +87661580222676988149508787898727904356830084933757595795694{x_6}^{65}\\ -99295130678732982697954163540091714286402009996178205412788{x_6}^{64}   
   \\
    -930119609210489893407678494136857339116003211787949649494086{x_6}^{63}  
              \end{eqnarray*}
     \begin{eqnarray*}  
     +8251624538636047639504707897514835165179814626622592612906383{x_6}^{62}\\ -40832553691879684939044909220462014642978264995169141443718513{x_6}^{61}\\ +153113844728828704258891375722446944761321837963396088913678084{x_6}^{60}\\ -491014436822069242360009239669232171629406738219167642226922097{x_6}^{59}\\ +1486167275729136237725219222922695061597100791853410803505204530{x_6}^{58}\\ -4507398440013026996842432323631370556426336375958748492115236727{x_6}^{57}\\ +13403136215594023614151820862044576128237228016726350101925041588{x_6}^{56}\\ -36271125718390133054894878722158117315551760241836497851870273093{x_6}^{55}\\ +83047794002702199629819182231616589211956031126634402375905909446{x_6}^{54}\\ -150193325304181943384609084747614184395393721106133360789756467387{x_6}^{53}\\ +188720043120217542280412396780001736455970138521743349912578548284{x_6}^{52}\\ -76741676617349091170093918618965231261859264468609973738315278031{x_6}^{51}\\ -349518040230090282264003239755015528354901423481817134201855637044{x_6}^{50}\\ 
  +1159869423644999127595987404946998669981199498972392560506911212007{x_6}^{49}\\ -2133935440931377513253465270342417844656186357247004386125088676980{x_6}^{48}\\ +2679153098878858240105331873966059777867005041600515836218049666291{x_6}^{47}\\ -2058237669308851388130070447624724250332446226969413641626759839185{x_6}^{46}\\ -86499479106794410056521461155508940987878993824757555118076751608{x_6}^{45}\\ +3256570539222645939094151055641278779497867738115006803090527339472{x_6}^{44}\\ -6138330577768613114245899426695789910243658269718926329419520882754{x_6}^{43}\\ +7304711956091898132457001381125874880268500470135196982359597863432{x_6}^{42}\\ -6138330577768613114245899426695789910243658269718926329419520882754{x_6}^{41}\\ +3256570539222645939094151055641278779497867738115006803090527339472{x_6}^{40}\\ -86499479106794410056521461155508940987878993824757555118076751608{x_6}^{39}\\ -2058237669308851388130070447624724250332446226969413641626759839185{x_6}^{38}\\ +2679153098878858240105331873966059777867005041600515836218049666291{x_6}^{37}\\ 
-2133935440931377513253465270342417844656186357247004386125088676980{x_6}^{36}  \nonumber \\ +1159869423644999127595987404946998669981199498972392560506911212007{x_6}^{35}\nonumber \\ -349518040230090282264003239755015528354901423481817134201855637044{x_6}^{34}\nonumber \\ -76741676617349091170093918618965231261859264468609973738315278031{x_6}^{33} \nonumber \\+188720043120217542280412396780001736455970138521743349912578548284{x_6}^{32}\nonumber \\-150193325304181943384609084747614184395393721106133360789756467387{x_6}^{31}\nonumber \\+83047794002702199629819182231616589211956031126634402375905909446{x_6}^{30}\nonumber \\-36271125718390133054894878722158117315551760241836497851870273093{x_6}^{29}\nonumber \\+13403136215594023614151820862044576128237228016726350101925041588{x_6}^{28}\nonumber \\-4507398440013026996842432323631370556426336375958748492115236727{x_6}^{27}\nonumber \\+1486167275729136237725219222922695061597100791853410803505204530{x_6}^{26}\nonumber \\-491014436822069242360009239669232171629406738219167642226922097{x_6}^{25}\nonumber 
        \end{eqnarray*}
     \begin{eqnarray}  
 +153113844728828704258891375722446944761321837963396088913678084{x_6}^{24}\nonumber 
 \\
 -40832553691879684939044909220462014642978264995169141443718513{x_6}^{23}\nonumber 
\\
 +8251624538636047639504707897514835165179814626622592612906383{x_6}^{22}\nonumber \\-930119609210489893407678494136857339116003211787949649494086{x_6}^{21}\nonumber \\-99295130678732982697954163540091714286402009996178205412788{x_6}^{20}\nonumber \\+87661580222676988149508787898727904356830084933757595795694{x_6}^{19}\nonumber \\-30615636009083905408974520798029796519322538407366673033114{x_6}^{18}\nonumber \\+8633720294376050564114347614772621001678803290761747417964{x_6}^{17}\nonumber \\-2039563706481217288587171432534950085804664817997809025036{x_6}^{16}\nonumber \\
+338557565736435688036663787252532332747532110554214415660{x_6}^{15}\nonumber \\-15498337035862631939271985201881911221051895322143942580{x_6}^{14}\nonumber \\-12482442118497915371064225556847850469474467891986831088{x_6}^{13}\nonumber \\+5328035280913750813152995811516544281680919577481605008{x_6}^{12}\nonumber \\            
 -1206001709632090015722926074131478604699691495585639240{x_6}^{11}\nonumber \\+126771120189744874933437937748815475102249501093644584{x_6}^{10}\nonumber \\+15784456131787043608289459347642059765045090310138944{x_6}^{9}\nonumber \\-8253365109064021447538266895144369502847176547816896{x_6}^{8}\nonumber \\+1146798966623805392526876699318358505159292698566560{x_6}^{7}\nonumber \\+16203553461279979489585518066332806201459880764320{x_6}^{6}\nonumber \\-23920323432867365791709113098968618521666096472928{x_6}^{5}\nonumber \\+2178950793290897699447263319032587719040052785696{x_6}^{4}\nonumber \\+74794100641615694951351011555563273206842944000{x_6}^{3}\nonumber \\-18883997141974881738092861801785333252894800000{x_6}^{2}\nonumber \\+472980983168555664602145477565706833680000000{x_6} \nonumber \\
      +19570190016315141603381713123537334000000000) = 0.  \label{eq84}
       \end{eqnarray}
    Moreover,   by examining the other elements of  the obtained Gr\"obner basis,  we see that the other variables $ x_2, x_3, x_4, x_5$ can be expressed by polynomials of $x_6$ with degree 83. 
    
    For the solutions $ ({x_6}-3) ({x_6}-2) (2 {x_6}-3) (2 {x_6}-1) (3 {x_6}-2) (3 {x_6}-1)=0$, we get 
    systems of solutions of the equation (\ref{eq84}) as follows: 
     \begin{eqnarray*} 
     & & x_6 = 3,  \quad x_5 = 2,  \quad x_4 =  \frac{5}{3}, \quad x_3 = \frac{4}{3}, \quad x_2 = \frac{1}{3}, 
     \\ & & x_6 = 2,  \quad x_5 = 3,  \quad x_4 =  \frac{5}{3}, \quad x_3 = \frac{1}{3}, \quad x_2 = \frac{4}{3}, 
     \\ & &  x_6 = \frac{3}{2},  \quad x_5 = \frac{1}{2},  \quad x_4 =  \frac{2}{3}, \quad x_3 = \frac{5}{6}, \quad x_2 = \frac{1}{6}, 
     \\ & & x_6 =  \frac{1}{2},  \quad x_5 = \frac{3}{2},  \quad x_4 =  \frac{2}{3}, \quad x_3 = \frac{1}{6}, \quad x_2 = \frac{5}{6},  
     \\ & &  x_6 = \frac{2}{3},  \quad x_5 = \frac{1}{3},  \quad x_4 =  \frac{1}{9}, \quad x_3 = \frac{5}{9}, \quad x_2 = \frac{4}{9}, 
     \\ & &  x_6 = \frac{1}{3},  \quad x_5 = \frac{2}{3},  \quad x_4 = \frac{1}{9}, \quad x_3 = \frac{4}{9}, \quad x_2 = \frac{5}{9}.  
      \end{eqnarray*} 
      Note that these are six K\"ahler-Einstein metrics in Theorem \ref{G2KE}. 
      
 Now,  by  solving  equation (\ref{eq84})  of the part  of degree 84 numerically,   we obtain 14 positive solutions  which are    approximately given by 
 \begin{eqnarray*}  & & 
  x_6 \approx 0.1101296649906623,  \quad  x_6 \approx  0.1276467609933986,  \quad   x_6 \approx 0.1654266507070432, \\
& &     x_6 \approx  0.2010643285289733,  \quad  
  x_6 \approx  0.3065328288396123,  \quad    x_6 \approx  0.5181203151843693,\\
 & &   x_6 \approx  0.5477334830916693,  \quad   x_6 \approx  1.82570544045531482,  \quad   x_6 \approx  1.93005363946047411,  
\\ 
             & & x_6 \approx  3.26229332037786929,  \quad   x_6 \approx  4.97353263662529741,  \quad    x_6 \approx  6.04497519429874693, \\ & &  x_6 \approx  7.83411966130276958,   \quad  x_6 \approx  9.08020559296887189.
  \end{eqnarray*} 
  To get the solutions of the equations (\ref{eqgeneral}) for variables $ x_2, x_3, x_4, x_5$ corresponding to the solution $x_6$, we 
 substitute these values for $x_6$ into the expressions of polynomials of $x_6$ with degree 83. Then we get  systems of solutions which are  approximately given by 
     \begin{eqnarray*}  & & 
  x_6 \approx 0.11013, \quad x_5 \approx 0.547733, \quad x_4 \approx 1.61358, \quad x_3 \approx 0.399131, \quad
   x_2 \approx -0.277481, 
  \\  & & 
   x_6 \approx 0.127647, \quad x_5 \approx -0.775539, \quad
   x_4 \approx 0.202709, \quad x_3 \approx 1.7601, \quad x_2 \approx -0.203265, 
   \\ 
              & & x_6 \approx 0.165427, \quad
   x_5 \approx -0.021892, \quad x_4 \approx 0.308989, \quad x_3 \approx 0.00455279, \quad
   x_2 \approx 0.5435, 
   \\  
            & & x_6 \approx 0.201064, \quad x_5 \approx 1.82571, \quad x_4 \approx 0.728695, \quad
   x_3 \approx 2.94591, \quad x_2 \approx -0.506599, 
   \\  & & x_6 \approx 0.306533, \quad x_5 \approx -1.52438, \quad
    x_4 \approx 0.207857, \quad x_3 \approx 1.64949, \quad x_2 \approx 5.33389, 
    \\  & & x_6 \approx 0.51812, \quad
   x_5 \approx -0.100239, \quad x_4 \approx -0.120371, \quad x_3 \approx -2.58645, \quad
   x_2 \approx -0.539579, 
   \\  & & x_6 \approx 0.547733, \quad x_5 \approx 0.11013, \quad x_4 \approx 1.61358, \quad
   x_3 \approx -0.277481, \quad x_2 \approx 0.399131, 
   \\  & & x_6 \approx 1.82571, \quad x_5 \approx 0.201064, \quad
    x_4 \approx 0.728695, \quad x_3 \approx -0.506599, \quad x_2 \approx 2.94591,
    \\  & & x_6 \approx 1.93005, \quad
    x_5 \approx -0.193467, \quad x_4 \approx -1.04142, \quad x_3 \approx -4.99198, \quad
   x_2 \approx -0.232323,
   \\  & & x_6 \approx 3.26229, \quad x_5 \approx -4.97297, \quad x_4 \approx 17.4007, \quad
   x_3 \approx 5.38113, \quad x_2 \approx 0.678092, 
   \\  & & x_6 \approx 4.97353, \quad x_5 \approx 9.08021, \quad
   x_4 \approx -2.51959, \quad x_3 \approx 14.6516, \quad x_2 \approx 3.62419,
   \\  & & x_6 \approx 6.04498, \quad
   x_5 \approx -0.132336, \quad x_4 \approx 3.28544, \quad x_3 \approx 0.0275215, \quad
   x_2 \approx 1.86783, 
   \\  & & x_6 \approx 7.83412, \quad x_5 \approx -6.07566, \quad x_4 \approx -1.5924, \quad
   x_3 \approx 13.7889, \quad x_2 \approx 1.58805, 
   \\  & & x_6 \approx 9.08021, \quad x_5 \approx 4.97353, \quad
   x_4 \approx -2.51959, \quad x_3 \approx 3.62419, \quad x_2 \approx 14.6516.  \end{eqnarray*} 
   
   Note that at least one of $x_i$ for these solutions is negative. Thus we have no invariant Einstein metrics for these cases.

\end{document}